\documentclass[11pt]{amsart}
\usepackage{amsmath}
\usepackage{amsfonts}
\usepackage{amssymb}
\usepackage{enumitem}
\usepackage{esint}
\newtheorem{teo}{Theorem}[section]
\newtheorem{lem}[teo]{Lemma}

\newtheorem{prop}[teo]{Proposition}

\newcommand\restrict[1]{\raisebox{-.5ex}{$|$}_{#1}}

\theoremstyle{remark}
\newtheorem{oss}[teo]{Remark}

\theoremstyle{definition}
\def \R {{\mathbb {R}}}

\DeclareMathOperator*{\diver}{div}

\newcommand{\average}{{\mathchoice {\kern1ex\vcenter{\hrule height.4pt
width 6pt
depth0pt} \kern-9.7pt} {\kern1ex\vcenter{\hrule height.4pt width 4.3pt
depth0pt}
\kern-7pt} {} {} }}

\def\p{\partial}

\def\vint_#1{\mathchoice%
          {\mathop{\kern 0.2em\vrule width 0.6em height 0.69678ex depth -0.58065ex
                  \kern -0.8em \intop}\nolimits_{\kern -0.4em#1}}%
          {\mathop{\kern 0.1em\vrule width 0.5em height 0.69678ex depth -0.60387ex
                  \kern -0.6em \intop}\nolimits_{#1}}%
          {\mathop{\kern 0.1em\vrule width 0.5em height 0.69678ex
              depth -0.60387ex
                  \kern -0.6em \intop}\nolimits_{#1}}%
          {\mathop{\kern 0.1em\vrule width 0.5em height 0.69678ex depth -0.60387ex
                  \kern -0.6em \intop}\nolimits_{#1}}}
\def\vintslides_#1{\mathchoice%
          {\mathop{\kern 0.1em\vrule width 0.5em height 0.697ex depth -0.581ex
                  \kern -0.6em \intop}\nolimits_{\kern -0.4em#1}}%
          {\mathop{\kern 0.1em\vrule width 0.3em height 0.697ex depth -0.604ex
                  \kern -0.4em \intop}\nolimits_{#1}}%
          {\mathop{\kern 0.1em\vrule width 0.3em height 0.697ex depth -0.604ex
                  \kern -0.4em \intop}\nolimits_{#1}}%
          {\mathop{\kern 0.1em\vrule width 0.3em height 0.697ex depth -0.604ex
                  \kern -0.4em \intop}\nolimits_{#1}}}

\newcommand{\kintint}[2]{\mathchoice%
          {\mathop{\kern 0.2em\vrule width 0.6em height 0.69678ex depth -0.58065ex
                  \kern -0.8em \intop}\nolimits_{\kern -0.45em#1}^{#2}}%
          {\mathop{\kern 0.1em\vrule width 0.5em height 0.69678ex depth -0.60387ex
                  \kern -0.6em \intop}\nolimits_{#1}^{#2}}%
          {\mathop{\kern 0.1em\vrule width 0.5em height 0.69678ex depth -0.60387ex
                  \kern -0.6em \intop}\nolimits_{#1}^{#2}}%
          {\mathop{\kern 0.1em\vrule width 0.5em height 0.69678ex depth -0.60387ex
                  \kern -0.6em \intop}\nolimits_{#1}^{#2}}}

\renewcommand{\div}{\operatorname{div}}

\makeatletter
\def\cleardoublepage{\clearpage\if@twoside \ifodd\c@page\else
\hbox{}
\thispagestyle{empty}
\newpage
\if@twocolumn\hbox{}\newpage\fi\fi\fi}
\makeatother
\title{A new glance to the Alt-Caffarelli-Friedman monotonicity formula}
\author{Fausto Ferrari}
\address{Fausto Ferrari: Dipartimento di Matematica\\ Universit\`a di Bologna\\ Piazza di Porta S.Donato 5\\ 40126, Bologna-Italy}
\email{fausto.ferrari@unibo.it }
\author{Nicol\`o Forcillo}
\address{Nicol\`o Forcillo: Dipartimento di Matematica\\ Universit\`a di Bologna\\ Piazza di Porta S.Donato 5\\ 40126, Bologna-Italy}
\email{nicolo.forcillo2@unibo.it }
\thanks{F.F.  and N.F. are partially supported by INDAM-GNAMPA-2019 projet: {\it Propriet\`a di regolarit\`a delle soluzioni viscose con applicazioni a problemi di frontiera libera.}}
\date{\today}

\begin{document}

\begin{abstract}
In this paper we revisit the proof of the Alt-Caffarelli-Friedman monotonicity formula. Then, in the framework of the Heisenberg group, we discuss the existence of an analogous monotonicity formula introducing a necessary condition for its existence, recently proved in \cite{FeFo}.
\end{abstract}
\maketitle
%Dedicated to Sandro Salsa, with respect and admiration: {\it\dots to a man who is used to cause events.}
\section{Introduction}
%In this paper we examine the structure of the functions that satisfy the following equation
%$$
%\left\{\begin{array}{ll}
%\Delta_{\mathbb{H}^1}u=0,& \mathcal{C}\cap B_R^{\mathbb{H}^1}(0)\\
%u=0,&\partial  \mathcal{C}\cap B_R^{\mathbb{H}^1}(0)
%\end{array}
%\right.
%$$ 
%where $\mathcal{C}\subset \mathbb{H}^1$ is an Euclidean cone with vertex at $(0,0,0)\in \mathbb{H}^1$ in the Heisenberg group $\mathbb{H}^1,$ and  $$B_R^{\mathbb{H}^1}(0)=\{(x,y,t)\in \mathbb{H}^1:\quad (x^2+y^2)^2+t^2<R^4 \}$$ is the Koranyi ball  centered at $(0,0,0)$ of radius $R.$ As a byproduct of this research we obtain some partial results  concerning an Alt-Caffarelli-Friedman monotonicity formula in the Heisenberg group.

The Alt-Caffarelli-Friedman monotonicity formula was introduced in \cite{ACF} as a fundamental tool for studying the main properties of the solutions of two-phase free boundary problems.

Roughly saying, following \cite{ACF}, the result says that there exists $r_0>0$ such that for every non-negative $u_1,u_2\in C(B_1(0))\cap H^1(B_1(0)),$ if $0\in \mathcal{F}(u_i),$ $\Delta u_i\geq 0,$ $i=1,2,$ $u_1(0)=u_2(0)=0$ and $u_1u_2=0$ in $B_1(0),$ where  $B_1(0)$ is the Euclidean ball centered at $0$ of radius $1$ in $\mathbb{R}^{n},$   then
\begin{equation}\label{definition-Phi}
\Phi(r):=r^{-4}\int_{B_r(0)}\frac{|\nabla u_1(x)|^2}{|x|^{n-2}}dx\int_{B_r(0)}\frac{|\nabla u_2(x)|^2}{|x|^{n-2}}dx
\end{equation}
is well defined, bounded and monotone increasing in $[0,r_0).$

 Alt, Caffarelli and Friedman used this result for proving the Lipschitz continuity
 of critical points of a functional like the following one
 \begin{equation}\label{energy}
 \mathcal{E}(v):=\int_{\Omega}\left(|\nabla v|^2+\chi_{\{v>0\}}\right)dx
\end{equation}
 defined on a set  $K\subset H^1(\Omega),$ where $\Omega\subset \mathbb{R}^n$ is a given bounded open set and $K$ is determined by some known conditions on $v$ given on $\partial \Omega,$ where $\chi_{\{v>0\}}$ denotes, as usual, the characteristic function of the set $\{v>0\}.$
 
The critical points of the previous functional $\mathcal{E}$ satisfy the following two-phase free boundary problem
\begin{equation}\label{two_phase_classical}
\begin{cases}
\Delta u=0& \mbox{in }\Omega^+(u):= \{x\in \Omega:\hspace{0.1cm} u(x)>0\},\\
\Delta u=0& \mbox{in }\Omega^-(u):=\mbox{Int}(\{x\in \Omega:\hspace{0.1cm} u(x)\leq 0\}),\\
|\nabla u^+|^2-|\nabla u^-|^2=1&\mbox{on }\mathcal{F}(u):=\partial \Omega^+(u)\cap \Omega,
\end{cases}
\end{equation}
see \cite{ACF}.
Thus, solutions of \eqref{two_phase_classical} satisfy, at least in a "weak" sense,
the following property: for every $P\in\mathcal{F}(u)$  
\[
(u_\nu^+(P))^2(u_\nu^-(P))^2=\lim_{r\to0^+}\Phi(r)\leq C,
\]
where $u^+:=\sup\{u,0\},$ $u^-:=\sup\{-u,0\},$  $\nu$ is the unit vector, pointing inside $\Omega^+(u)$ 
%for $u^+$
 at $P\in \mathcal{F}(u) $ and inside $\Omega^-(u)$ 
% for $u^-$ 
 at $P\in \mathcal{F}(u)$ when this makes sense in the smooth case. See \cite{CS} for a more general viscosity meaning.

Hence, if one of the two phases, let say $u^-,$ is sufficiently regular at $P\in \mathcal{F}(u),$ see \cite{GT}, then by Hopf maximum principle it results $u^-_\nu(P)>0$ so that, as a by-product, $u^+_\nu(P)$ has to be bounded. In this way, the solutions of the free boundary problem are globally Lipschitz.

After \cite{ACF} many other important papers on this topic appeared. We remind some of them, without pretending of citing all the literature about this topic. In \cite{C3} it was proved that monotonicity formula holds for linear uniformly elliptic operators in divergence form with H\"older continuous coefficients, in \cite{CJK} a formula for non-homogeneous free boundary problems was discovered,  in \cite{TZ} the Riemannian case was treated, while in \cite{MP}  the non-divergence form case has been faced. Some very partial results have been obtained also in the nonlinear case in lower dimension: see \cite{DiKar} for the $p-$Laplace case.

Moreover, this formula became popular and popular for other applications as well. Among them, 
there are further two-phase problems, see \cite{C1} for the elliptic homogeneous case, \cite{ACS1} and \cite{FS_p} for the parabolic homogeneous setting, and \cite{DFS_APDE} for the elliptic linear non-homogeneous problems. In addition we also recall some segregation problems, see for instance: \cite{NTV}, \cite{Q}, \cite{TVZ} and \cite{TTV}. In this way,  during the last decade, the Alt-Caffarelli-Friedman monotonicity formula has quickly increased its importance in literature. 

The existence of such a tool for elliptic degenerate operators, for instance sublaplacians on groups, as far as we know, has not yet been understood. Anyhow, concerning other similar formulas about sublaplacians we find in literature some important contributions, see \cite{Garofalo_Lanconelli} and in particular \cite{Garofalo_Rotz}, where the authors deal with the frequency function of Almgren in Carnot groups. Moreover, see \cite{DGS} and \cite{DGP_ob} for further papers in non-commutative setting dealing with other free boundary problems, namely the obstacle problem. 

We also gently warn the reader about  the existence of results about two-phase problems in the Heisenberg group, like  \cite{FeVa} and \cite{Fe} in particular, where the following  parallel version of the Euler equations (\ref{two_phase_classical}), of a two-phase problem in this non-commutative framework, has been achieved:
%  descending by the functiona $\mathcal{E}_{\mathbb{H}^n}$ ,
\begin{equation}\label{two_phase_Heisenberg}
\begin{cases}
\Delta_{\mathbb{H}^n} u=0& \mbox{in }\Omega^+(u):= \{x\in \Omega:\hspace{0.1cm} u(x)>0\}\\
\Delta_{\mathbb{H}^n} u=0& \mbox{in }\Omega^-(u):=\mbox{Int}(\{x\in \Omega:\hspace{0.1cm} u(x)\leq 0\})\\
|\nabla_{\mathbb{H}^n} u^+|^2-|\nabla_{\mathbb{H}^n} u^-|^2=1&\mbox{on }\mathcal{F}(u):=\partial \Omega^+(u)\cap \Omega.
\end{cases}
\end{equation}
% see \cite{Fe}. 
 In the Section \ref{secondsection} of this paper, we shall introduce the main notation that we need for working on this subject.
%  in the Heisenberg group,
%   anyhow 
 Nevertheless, we ask to the reader that is not customary with this language to continue to follow this colloquial presentation having in mind that, in the Heisenberg group, there exists a natural translation of the classical Euclidean tools in terms of parallel {\it intrinsic notions} in the non-commutative structure  $\mathbb{H}^n$.
%  given by
%  framework given
%   by
Hence, what we are going to discuss in a while in this introduction, it should be easily interpreted by all.
%, for understanding what we have proved.

We remark that, in this particular non-commutative context, the gradient jump $|\nabla u^+|^2-|\nabla u^-|^2=1$ now is governed by the jump of the horizontal gradient $\nabla_{\mathbb{H}^n}$ of the solutions $u$ of (\ref{two_phase_Heisenberg}) in the Heisenberg group $\mathbb{H}^n.$  As a first consequence, in this degenerate case associated with the sublaplacian $\Delta_{\mathbb{H}^n}$, a new geometric problem, that in the Euclidean two-phase problem did not exist, now appears. In fact, since classical smooth free boundaries of (\ref{two_phase_Heisenberg}), in principle, might have {\it characteristic points}, then the jump of the horizontal gradient of $u$ on $\mathcal{F}(u)$ could be no satisfied pointwise, because the horizontal gradient vanishes on characteristic points, see Section \ref{secondsection}.
 
It is also worthwhile to recall that it has been already proved that a minimum $u$ of  the functional 
 \[
 \mathcal{E}_{\mathbb{H}^n}(v):=\int_{\Omega}\left(|\nabla_{\mathbb{H}^n} v|^2+\chi_{\{v>0\}}\right)dx,
\]
$\Omega \subset\mathbb{H}^n,$ see  Section 3 in \cite{FeVa},
 is endowed by a locally bounded horizontal gradient $\nabla_{\mathbb{H}^n} u$ and moreover that every minimum $u$ satisfies $\Delta_{\mathbb{H}^n} u=0\:\: \mbox{in }\Omega^+(u),$ as well as $\Delta_{\mathbb{H}^n} u=0\:\: \mbox{in }\Omega^-(u),$ even if no word has been spent about the behavior of the free boundary of these local minima. Indeed,  an alternatively way of proving that a local minimum of the functional $\mathcal{E}_{\mathbb{H}^n}$ is intrinsically Lipschitz, instead of using the monotonicity formula, has been shown in \cite{FeVa}.   

%Thus, with this paper we focus our attention to an Alt-Caffarelli-Friedman monotonicity formula in the simplest Heisenberg group $\mathbb{H}^1.$

The proof of the monotonicity formula in the Euclidean framework is quite long and based on many highly non-trivial results. Thus, we like to revisit it in Section \ref{Euclidean_setting}, by commenting the key points of the proof and then focusing our attention to the parallel steps that we would need to prove in the Heisenberg group, including the statement of our following main result proved in \cite{FeFo} as well.

In order to reach our goal, let us introduce the following family of functionals depending on a real number $\beta>0:$  
\begin{equation}\label{monofondformula1}
J_{\beta,\mathbb{H}^1}(r)= r^{-\beta}\int_{B_r^{{\mathbb{H}^1}}(0)}\frac{\mid\nabla_{\mathbb{H}^1} u_1\mid^2}{|\zeta |_{\mathbb{H}^1}^{2}}d\zeta\int_{B_r^{{\mathbb{H}^1}}(0)}\frac{\mid\nabla_{\mathbb{H}^1} u_2\mid^2}{|\zeta|_{\mathbb{H}^1}^{2}}d\zeta.
\end{equation}

Following the main steps of the Euclidean proof, in \cite{FeFo} we proved the following result as a corollary of an estimation of the first eigenvalue of an operator defined on the boundary of the Koranyi ball of radius one. In fact, as people who work with Heisenberg group stuff well know, this set takes the place of the boundary of the classical Euclidean ball of radius one, when we need to work with the fundamental solution of the sublaplacian $\Delta_{\mathbb{H}^1},$ see \cite{Folland}.  

\begin{teo}\label{corcor1} If there exists a positive number $\beta$ for which $J_{\beta,\mathbb{H}^1}$
%$$
%J_{\beta,\mathbb{H}^1}(r)=r^{-\beta}\int_{B_r^{{\mathbb{H}^1}}(0)}\frac{\mid\nabla_{\mathbb{H}^n} u_1\mid^2}{|\zeta |_{\mathbb{H}^1}^{2}}d\zeta\int_{B_r^{{\mathbb{H}^1}}(0)}\frac{\mid\nabla_{\mathbb{H}^1} u_2\mid^2}{|\zeta |_{\mathbb{H}^1}^{2}}d\zeta
%$$
 is monotone 
 for every $u_1, u_2\in H_{\mathbb{H}^1}^1(B_1^{{\mathbb{H}^1}}(0))$, such that $\Delta_{\mathbb{H}^1}u_i\geq 0,$ $u_i(0)=0,$ $i=1,2$ and $u_1u_2=0,$   then $\beta\leq 4.$
\end{teo}
We stated this result in the first Heisenberg group only, because we did not prove a monotonicity formula for all the Heisenberg groups, but simply we have proved that if this formula holds in the non-commutative framework given by $\mathbb{H}^1$, then the right exponent $\beta$ has to be smaller or equal than $4.$ The proof in higher Heisenberg groups requires more computations, but it may be obtained with some further efforts, that we do not discuss here, following the same ideas. On the other hand,  the breakthrough that we would need for concluding that, at least in $\mathbb{H}^1,$ the sharp exponent $\beta$ is exactly $4$,  depends on a long standing open question.
 In fact the best profile of the set that realizes the equality in the isoperimetric inequality in the Heisenberg group (and as a byproduct the descendant Polya-Sz\"ego inequality on the surface of the Koranyi ball of radius one) is still open, see \cite{CDPT} for an introduction to this problem. So that, considering  previous arguments, we have decided to state our result only in $\mathbb{H}^1.$  We shall discuss this part in Section \ref{lackofisoperimetric}. In the remaining Section \ref{basictoolsinHeisenberg}, we describe the main tools we need for obtaining the key estimate on the Rayleigh quotient in $\mathbb{H}^1$, see \cite{FeFo} for the details.

\section{The Euclidean setting}\label{Euclidean_setting}
In this section, following the original paper \cite{ACF}, and \cite{CS}, we try to focus on the main steps we need to achieve for proving the Alt-Caffarelli-Friedman monotonicity formula in the Euclidean setting.

 After a straightforward differentiation, it results 
\begin{equation}\label{derivative-Phi}
\Phi'(r)=I_{1}(r)I_{2}(r)r^{-5}\left(-4+r\left(\frac{I_{1}'}{I_{1}}+\frac{I_{2}'}{I_{2}}\right)\right),
\end{equation}
where for $i=1,2:$
\[
I_i(r)=\int_{B_r(0)}\frac{|\nabla u_i(x)|^2}{|x|^{n-2}}dx.
\]
By a rescaling argument the problem may be reduced to 
\begin{equation}\label{derivative-Phi-rescaled}
\Phi'(r)=I_{1}(r)I_{2}(r)r^{-5}\left(-4+\frac{\displaystyle\int_{\partial B_1(0)}|\nabla u_1(x)|^2d\sigma}{\displaystyle\int_{ B_1(0)}\frac{|\nabla u_1(x)|^2}{|x|^{n-2}}dx}+\frac{\displaystyle\int_{\partial B_1(0)}|\nabla u_2(x)|^2d\sigma}{\displaystyle\int_{ B_1(0)}\frac{|\nabla u_2(x)|^2}{|x|^{n-2}}dx}\right).
\end{equation}
Precisely, we have
\[I_i(r)=\int_{B_r(0)}\frac{|\nabla u_i(x)|^2}{|x|^{n-2}}dx=\int_{B_1(0)}\frac{|\nabla u_i(ry)|^2}{|ry|^{n-2}}\hspace{0.05cm}r^n\hspace{0.05cm}dy=r^2\int_{B_1(0)}\frac{|\nabla u_i(ry)|^2}{|y|^{n-2}}\hspace{0.05cm}dy,\]
and
\begin{align*}
I_i(r)=\int_{B_r(0)}\frac{|\nabla u_i(x)|^2}{|x|^{n-2}}dx&=\int_0^r\left(\int_{\partial B_{\rho}(0)}\frac{|\nabla u_i(x)|^2}{|x|^{n-2}}\hspace{0.05cm}d\sigma(x)\right)d\rho=\int_0^r\left(\int_{\partial B_{1}(0)}\frac{|\nabla u_i(\rho y)|^2}{\rho^{n-2}}\hspace{0.05cm}\rho^{n-1}d\sigma(y)\right)d\rho\\
&=\int_0^r\rho \left(\int_{\partial B_{1}(0)}|\nabla u_i(\rho y)|^2\hspace{0.05cm}d\sigma(y)\right)d\rho,
\end{align*}
where here $y$ denotes the coordinates on $\partial B_1(0).$ Thus, we get
\begin{align*}
\frac{I_i'}{I_i}&=\frac{\displaystyle\frac{d}{d r}\int_0^r\rho \left(\int_{\partial B_{1}(0)}|\nabla u_i(\rho y)|^2\hspace{0.05cm}d\sigma(y)\right)d\rho}{\displaystyle r^2\int_{B_1(0)}\frac{|\nabla u_i(ry)|^2}{|y|^{n-2}}\hspace{0.05cm}dy}=\frac{r\displaystyle \int_{\partial B_{1}(0)}|\nabla u_i(r y)|^2\hspace{0.05cm}d\sigma(y)}{\displaystyle r^2\int_{B_1(0)}\frac{|\nabla u_i(ry)|^2}{|y|^{n-2}}\hspace{0.05cm}dy}\\
&=\frac{1}{r}\frac{\displaystyle \int_{\partial B_{1}(0)}|\nabla u_i(r y)|^2\hspace{0.05cm}d\sigma(y)}{\displaystyle \int_{B_1(0)}\frac{|\nabla u_i(ry)|^2}{|y|^{n-2}}\hspace{0.05cm}dy},
\end{align*}
which implies, if we define 
\[(u_i)_r(x)=\frac{u_i(rx)}{r},\quad x\in B_1,\]
that
\[\frac{I_i'}{I_i}=\frac{1}{r}\frac{\displaystyle \int_{\partial B_{1}(0)}|\nabla (u_i)_r( y)|^2\hspace{0.05cm}d\sigma(y)}{\displaystyle \int_{B_1(0)}\frac{|\nabla (u_i)_r(y)|^2}{|y|^{n-2}}\hspace{0.05cm}dy},\]
where $(u_i)_r$ is defined in $B_1(0).$ As a consequence, if we write $y=x$ and $(u_i)_r=u_i$ the last equality gives
\[r\frac{I_i'}{I_i}=\frac{\displaystyle \int_{\partial B_{1}(0)}|\nabla u_i( y)|^2\hspace{0.05cm}d\sigma(y)}{\displaystyle \int_{B_1(0)}\frac{|\nabla u_i(y)|^2}{|y|^{n-2}}\hspace{0.05cm}dy},\]
and so \eqref{derivative-Phi} becomes \eqref{derivative-Phi-rescaled}.\newline
Now, if 
$$
-4+\frac{\displaystyle\int_{\partial B_1(0)}|\nabla u_1(x)|^2d\sigma}{\displaystyle\int_{ B_1(0)}\frac{|\nabla u_1(x)|^2}{|x|^{n-2}}dx}+\frac{\displaystyle\int_{\partial B_1(0)}|\nabla u_2(x)|^2d\sigma}{\displaystyle\int_{ B_1(0)}\frac{|\nabla u_2(x)|^2}{|x|^{n-2}}dx}\geq 0
$$
then, from \eqref{derivative-Phi-rescaled}, $\Phi'(r)\geq 0.$
Hence, in order to prove that previous inequality holds, the following ratios
$$
J_i(r):=\frac{\displaystyle\int_{\partial B_1(0)}|\nabla u_i(x)|^2d\sigma}{\displaystyle\int_{ B_1(0)}\frac{|\nabla u_i(x)|^2}{|x|^{n-2}}dx},
$$
for $i=1,2,$ have to be estimated. 

Since the gradient may split in two orthogonal parts involving the radial part and the tangential part, respectively denoted by $\nabla^\rho u_i$ and $\nabla^\theta u_i,$ it results
$$
|\nabla u_i(x)|^2=|\nabla^\rho u_i(x)|^2+|\nabla^\theta u_i(x)|^2.
$$
Then, we can rewrite $J_i$ as
\begin{equation}\label{rewriting-of-J}
J_i(r)=\frac{\displaystyle\int_{\partial B_1(0)}\left(|\nabla^{\rho} u_i(x)|^2+|\nabla^{\theta} u_i(x)|^2\right)\hspace{0.05cm}d\sigma}{\displaystyle\int_{ B_1(0)}\frac{|\nabla u_i(x)|^2}{|x|^{n-2}}dx}.
\end{equation}
At this point, we estimate the numerator and denominator of \eqref{rewriting-of-J} separately.\newline
As regards the numerator, we define first
\[\lambda(\Gamma_i):=\inf_{v\hspace{0.025cm}\in\hspace{0.025cm} H_0^1(\Gamma_i)}\frac{\displaystyle\int_{\Gamma_i}|\nabla^\theta v(x)|^2d\sigma}{\displaystyle\int_{\Gamma_i}v(x)^2d\sigma},\]
where 
\[\Gamma_i:=\{x\in \partial B_1(0):\hspace{0.1cm} u_i(x)>0\}\]
and $\lambda(\Gamma_i),$ $i=1,2,$ is the Rayleigh quotient. By the definition of $\lambda(\Gamma_i),$ we thus obtain, for every $\beta_i\in (0,1),$
\begin{align*}
&\int_{\partial B_{1}(0)}\left|\nabla^{\theta}u_i(x)\right|^2 d\sigma=\int_{\Gamma_i}\left|\nabla^{\theta}u_i(x)\right|^2 d\sigma\geq \lambda(\Gamma_i)\int_{\Gamma_i}u_i(x)^2d\sigma\\
&=(1-\beta_i+\beta_i)\lambda(\Gamma_i)\int_{\Gamma_i}u_i(x)^2d\sigma=\beta_i\lambda(\Gamma_i)\int_{\Gamma_i}u_i(x)^2d\sigma+(1-\beta_i)\lambda(\Gamma_i)\int_{\Gamma_i}u_i(x)^2d\sigma,
\end{align*}
hence, by Cauchy inequality, we have
\begin{equation}\label{lower-bound-numerator}
\begin{split}
&\int_{\partial B_{1}(0)}\left(|\nabla^{\rho}u_i(x)|^2+|\nabla^{\theta}u_i(x)|^2\right)\hspace{0.05cm}d\sigma\geq \int_{ \Gamma_i}|\nabla^{\rho}u_i(x)|^2d\sigma+\beta_i\lambda(\Gamma_i)\int_{\Gamma_i}u_i(x)^2d\sigma\\
&+(1-\beta_i)\lambda(\Gamma_i)\int_{\Gamma_i}u_i(x)^2d\sigma\geq 2 \left(\int_{ \Gamma_i}|\nabla^{\rho}u_i(x)|^2d\sigma\right)^{1/2}\left(\beta_i\lambda(\Gamma_i)\int_{\Gamma_i}u_i(x)^2d\sigma\right)^{1/2}\\
&+(1-\beta_i)\lambda(\Gamma_i)\int_{\Gamma_i}u_i(x)^2d\sigma.
\end{split}
\end{equation}
Concerning the denominator, instead, we compute
\[\Delta(u_i^2)=\sum_{j=1}^n\frac{\partial^2}{\partial x_j^2}\left(u_i^2\right)=\sum_{j=1}^n\frac{\partial}{\partial x_j}\left(2u_i\frac{\partial u_i}{\partial x_j}\right)=2\left(\left|\nabla u_i\right|^2+u_i\Delta u_i\right)\geq 2\left|\nabla u_i\right|^2,\]
since $u_i\Delta u_i\geq 0$ by the assumptions on $u_i.$ 

Consequently, we achieve 
the following estimate: 
\begin{equation}\label{upper-bound-denominator}
\int_{ B_1(0)}\frac{|\nabla u_i(x)|^2}{|x|^{n-2}}dx\leq \left(\int_{ \Gamma_i} \left|\nabla^\rho u_i(x)\right|^2d\sigma\right)^{\frac{1}{2}}\left(\int_{ \Gamma_i}u_i^2(x)\hspace{0.05cm}d\sigma\right)^{\frac{1}{2}}+\frac{n-2}{2}\int_{\Gamma_i }u_i^2(x)\hspace{0.05cm}d\sigma. 
\end{equation}

In fact, previous inequality follows after an integration by parts, using the facts that $\left|x\right|^{2-n}$ is, up to a multiplicative constant, the fundamental solution of $\Delta$ and $0\in \mathcal{F}(u_i),$ $i=1,2,$ and by H\"older inequality because:
\begin{align*}
&\int_{ B_1(0)}\frac{|\nabla u_i(x)|^2}{|x|^{n-2}}dx\leq \frac{1}{2}\int_{ B_1(0)}\Delta\left(u_i^2\right)(x)|x|^{2-n}dx=\frac{1}{2}\int_{ B_1(0)}\div\left(\left|x\right|^{2-n}\nabla (u_i^2)(x)\right)dx\\
&-\int_{ B_1(0)}\nabla (\left|x\right|^{2-n})\cdot\nabla(u_i^2)(x)\hspace{0.05cm}dx=\frac{1}{2}\bigg(\int_{\partial B_{1}(0)}\left|x\right|^{2-n}\nabla (u_i^2)(x)\cdot \frac{x}{\left|x\right|}d\sigma-\int_{ B_1(0)}\div\left(u_i^2(x)\nabla (\left|x\right|^{2-n})\right)dx\\
&+\int_{ B_1(0)}u_i^2(x)\Delta(\left|x\right|^{2-n})dx\bigg)
%=\frac{1}{2}\left(\int_{\Gamma_i}2u_i(x)\nabla u_i(x)\cdot \frac{x}{\left|x\right|}d\sigma-\int_{ \partial B_1(0)}u_i^2(x)\nabla\left(\left|x\right|^{2-n}\right)\cdot \frac{x}{\left|x\right|}d\sigma+(u_i(0))^2\right)
%\end{align*}
%\begin{align*}
%&=\frac{1}{2}\left(\int_{\Gamma_i}2u_i(x)\nabla^{\rho} u_i(x)\hspace{0.05cm}d\sigma-\int_{ \partial B_1(0)}u_i^2(x)\nabla\left(\left|x\right|^{2-n}\right)\cdot \frac{x}{\left|x\right|}d\sigma\right)\\
%&=\frac{1}{2}\left(\int_{\Gamma_i}2u_i(x)\nabla^{\rho} u_i(x)\hspace{0.05cm}d\sigma-\int_{\Gamma_i }u_i^2(x)(2-n)\left|x\right|^{1-n}\frac{x}{\left|x\right|}\cdot \frac{x}{\left|x\right|}d\sigma\right)\\&
=\frac{1}{2}\left(\int_{\Gamma_i}2u_i(x)\nabla^{\rho} u_i(x)\hspace{0.05cm}d\sigma+(n-2)\int_{\Gamma_i }u_i^2(x)\left|x\right|^{1-n}d\sigma\right)\\
&=\int_{\Gamma_i}u_i(x)\nabla^{\rho} u_i(x)\hspace{0.05cm}d\sigma+\frac{n-2}{2}\int_{\Gamma_i }u_i^2(x)d\sigma\leq \left(\int_{ \Gamma_i} \left|\nabla^\rho u_i(x)\right|^2d\sigma\right)^{\frac{1}{2}}\left(\int_{ \Gamma_i}u_i^2(x)\hspace{0.05cm}d\sigma\right)^{\frac{1}{2}}+\frac{n-2}{2}\int_{\Gamma_i }u_i^2(x)\hspace{0.05cm}d\sigma.
\end{align*}
Now, putting together \eqref{lower-bound-numerator} and \eqref{upper-bound-denominator}, we get, in view of \eqref{rewriting-of-J},
\begin{equation}\label{lower-bound-J-i}
\begin{split}
&
%\frac{\int_{\partial B_1(0)}\frac{|\nabla u_i(x)|^2}{|x|^{n-2}}d\sigma}{\int_{ B_1(0)}\frac{|\nabla u_i(x)|^2}{|x|^{n-2}}dx}
J_i(r)\geq \frac{2\left(\displaystyle\int_{\Gamma_i}|\nabla^\rho u_i(x)|^2d\sigma\right)^{\frac{1}{2}}\left(\displaystyle\int_{\Gamma_i}\beta_i \lambda(\Gamma_i)  u_i^2(x)\hspace{0.05cm}d\sigma\right)^{\frac{1}{2}}+(1-\beta_i)\lambda(\Gamma_i)\displaystyle\int_{\Gamma_i}u_i^2(x)\hspace{0.05cm}d\sigma}{\left(\displaystyle\int_{ \Gamma_i} \left|\nabla^\rho u_i(x)\right|^2d\sigma\right)^{\frac{1}{2}}\left(\displaystyle\int_{ \Gamma_i}u_i^2(x)\hspace{0.05cm}d\sigma\right)^{\frac{1}{2}}+\displaystyle\frac{n-2}{2}\displaystyle\int_{ \Gamma_i}u_i^2(x)\hspace{0.05cm}d\sigma},
\end{split}
\end{equation}
and setting $\xi_i=\left(\displaystyle\int_{ \Gamma_i} \left|\nabla^\rho u_i(x)\right|^2d\sigma\right)^{\frac{1}{2}}$ and $\eta_i=\left(\displaystyle\int_{\Gamma_i}  u_i^2(x)\hspace{0.05cm}d\sigma\right)^{\frac{1}{2}},$ it holds
\begin{equation*}
\begin{split}
&J_i(r)\geq\frac{2(\beta_i\lambda(\Gamma_i))^{\frac{1}{2}}\xi_i\eta_i+(1-\beta_i)\lambda (\Gamma_i)\eta_i^2}{\xi_i\eta_i+\frac{n-2}{2} \eta_i^2}= \frac{2(\beta_i\lambda(\Gamma_i))^{\frac{1}{2}}+(1-\beta_i)\lambda (\Gamma_i)\frac{\eta_i}{\xi_i}}{1+\frac{n-2}{2} \frac{\eta_i}{\xi_i}}\\
&\geq \inf_{z\geq 0} \frac{2(\beta_i\lambda(\Gamma_i))^{\frac{1}{2}}+(1-\beta_i)\lambda (\Gamma_i)z}{1+\frac{n-2}{2} z}=2\min\left\{\frac{\lambda (\Gamma_i)}{n-2}(1-\beta_i),(\beta_i\lambda(\Gamma_i))^{\frac{1}{2}}\right\}.
\end{split}
\end{equation*}
The last equality easily follows by elementary arguments.

Now, if it were possible to choose $\beta_i\in (0,1)$ in such a way that 
$$
\frac{\lambda (\Gamma_i)}{n-2}(1-\beta_i)=(\beta_i\lambda(\Gamma_i))^{\frac{1}{2}}
$$
we would realize, by denoting  $\alpha_i:=(\beta_i\lambda(\Gamma_i))^{\frac{1}{2}},$ that previous equation is satisfied if and only if
$$
 \alpha_i^2+(n-2)\alpha_i-\lambda(\Gamma_i)=0.
$$
On the other hand, since a function $u=\rho^{\alpha}g(\theta),$ $\theta:=(\theta_1,\dots,\theta_{n-1}),$ is harmonic in a cone determined by a domain $\Gamma$ whenever
$$
\rho^{\alpha-2}\left(( \alpha(\alpha-1)+\alpha (n-1))g(\theta)+\Delta_{\theta} g\right)=0,
$$
we deduce that there exists $\alpha_i$ such that
$$
\alpha_i(\alpha_i-1)+\alpha_i (n-1)=\lambda(\Gamma_i),
$$
namely
\[\alpha_i^2+(n-2)\alpha_i-\lambda(\Gamma_i)=0.\]
By the structure of the equation, it immediately comes out that there always exists a strictly positive solution $\alpha_i=\alpha_i(\Gamma_i),$ which is called the characteristic constant of $\Gamma_i.$ 

Therefore, we have to prove the existence of $\beta_i\in (0,1)$ such that
\begin{equation}\label{betaf}
\frac{-(n-2)+\sqrt{(n-2)^2+4\lambda(\Gamma_i)}}{2}=(\beta_i\lambda(\Gamma_i))^{\frac{1}{2}}.
\end{equation}
Specifically, \eqref{betaf} is equivalent to solve
$$
\frac{4\lambda(\Gamma_i)}{(n-2)+\sqrt{(n-2)^2+4\lambda(\Gamma_i)}}=2(\beta_i\lambda(\Gamma_i))^{\frac{1}{2}},
$$
that is
$$
\frac{2\lambda(\Gamma_i)^{\frac{1}{2}}}{(n-2)+\sqrt{(n-2)^2+4\lambda(\Gamma_i)}}=\beta_i^{\frac{1}{2}}.
$$
Since the continuous positive function defined in $[0,+\infty)$ as 
$$
z\to\frac{z}{(n-2)+\sqrt{(n-2)^2+z^2}}
$$
is strictly increasing, $\left(\frac{z}{(n-2)+\sqrt{(n-2)^2+z^2}}\right)(0)=0$ and $\sup\limits_{[0,+\infty)} \frac{z}{(n-2)+\sqrt{(n-2)^2+z^2}}=1,$ we conclude that for every $\lambda(\Gamma_i)>0,$ there exists $\beta_i$ such that
(\ref{betaf}) holds. In particular, we get
$$
\beta_i=\left(\frac{2\lambda(\Gamma_i)^{\frac{1}{2}}}{(n-2)+\sqrt{(n-2)^2+4\lambda(\Gamma_i)}}\right)^2.
$$
Hence, with previous choice of $\beta_i,$ if we denote
$$
\alpha_i:=\min\left\{\frac{\lambda (\Gamma_i)}{n-2}(1-\beta_i),(\beta_i\lambda(\Gamma_i))^{\frac{1}{2}}\right\},
$$
which is also the exponent corresponding to the eigenvalue given by the Rayleigh quotient $\lambda (\Gamma_i),$ we conclude that, whenever
\begin{equation}\label{belowestimate}
\alpha_1+\alpha_2\geq 2,
\end{equation}
then $\Phi'\geq 0.$

So, for completing this proof, we would need to know that (\ref{belowestimate}) holds.

To this end, by \cite{Sperner} we know that $\alpha_i(\Gamma_i)\geq \alpha_i(\Gamma_i^*),$ where $\Gamma_i^*\subset \partial B_1(0)$ is a spherical cap, namely a set of the form
\[\Gamma_i^*=\partial B_1(0)\cap \left\{x_n>s\right\},\quad -1<s<1,\] such that $\mathcal{H}^{n-1}(\Gamma_i)=\mathcal{H}^{n-1}(\Gamma_i^*).$ Here  $\mathcal{H}^{n-1}$ denotes the $(n-1)$-dimensional Hausdorff measure on $\partial B_1(0).$\newline
Precisely, \cite{Sperner} shows that if $u\in C^{\infty}(\partial B_1(0),\R),$ then 
\begin{equation}\label{result-theorem-sperner}
\begin{cases}
\displaystyle\int\limits_{\partial B_1(0)}\left\|\nabla u^*\right\|^pd\mathcal{H}^{n-1}\leq\displaystyle\int\limits_{\partial B_1(0)}\left\|\nabla u\right\|^pd\mathcal{H}^{n-1}& 1\leq p<\infty,\\
\left\|\nabla u^*\right\|_{L^{\infty}(\partial B_1(0))}\leq \left\|\nabla u\right\|_{L^{\infty}(\partial B_1(0))},
\end{cases}
\end{equation}
where $u^*$ is the symmetrized function of $u,$ depending only on the latitude of the argument. Moreover, we also have that $u_{\#}(\mathcal{H}^{n-1})\restrict{\mathcal{B}(\R)}=u^*_{\#}(\mathcal{H}^{n-1})\restrict{\mathcal{B}(\R)},$ that is the pushforward measures of $u$ and $u^*$ coincide in the Borel sets of $\R,$ which entails 
\begin{equation}\label{equality-integral-phi-u-phi-u-*}
\int\limits_{\partial B_1(0)}\phi\circ u\hspace{0.1cm}d\mathcal{H}^{n-1}=\int\limits_{\partial B_1(0)}\phi\circ u^*\hspace{0.1cm}d\mathcal{H}^{n-1},
\end{equation}
for any function $\phi:\R\to \R$ $\mu^*$-measurable, where $\mu^*$ is the outer measure defined on the power set $\mathbb{P}(\R)$ of $\R$ as 
\[\mu^*(F)=\inf\left\{\sum_{i=1}^{\infty}\mu(A_i):A_i\in \mathcal{B}(\R),F\subset \bigcup_{i=1}^{\infty}A_i\right\},\]
with $\mu=u_{\#}(\mathcal{H}^{n-1})\restrict{\mathcal{B}(\R)}=u^*_{\#}(\mathcal{H}^{n-1})\restrict{\mathcal{B}(\R)}$ and $F\in \mathbb{P}(\R).$ Hence, choosing $\phi=x^2$ in \eqref{equality-integral-phi-u-phi-u-*}, we obtain
\[\int\limits_{\partial B_1(0)} u^2\hspace{0.1cm}d\mathcal{H}^{n-1}=\int\limits_{\partial B_1(0)}(u^*)^2\hspace{0.1cm}d\mathcal{H}^{n-1},\]
which gives, together with \eqref{result-theorem-sperner}, $\lambda(\Gamma_i)\geq \lambda(\Gamma_i^*),$ and thus, using the expression of $\alpha_i(\Gamma_i),$ $\alpha_i(\Gamma_i)\geq \alpha_i(\Gamma_i^*),$ since $u^*$ is defined on $\Gamma_i^*,$ if $u$ is defined on $\Gamma_i.$ The fact that $\mathcal{H}^{n-1}(\Gamma_i)=\mathcal{H}^{n-1}(\Gamma_i^*)$ derives from a property of $u^*$ which says that
\[\mathcal{H}^{n-1}(u^{-1}[\rho,\infty))=\mathcal{H}^{n-1}((u^*)^{-1}[\rho,\infty)),\quad \forall \rho \in \R.\]
On the other hand, from \cite{FH} we achieve that $\alpha_i(\Gamma_i^*)\geq \psi(s_i),$ where $s_i=\frac{\mathcal{H}^{n-1}(\Gamma_i^*)}{\mathcal{H}^{n-1}(\p B_1(0))}$ and $\psi(s),$ $s\in(0,1),$ is convex and decreasing. In particular, $\psi(s)$ is defined as 
\begin{equation}\label{definition-psi}
\psi(s):=
\begin{cases}
\frac{1}{2}\log\dfrac{1}{4s}+\frac{3}{2},&0<s\leq\dfrac{1}{4},\\
2(1-s),&\dfrac{1}{4}\leq s<1.
\end{cases}
\end{equation}
Precisely, the proof of $\alpha_i(\Gamma_i^*)\geq \psi(s_i)$ is organized in some steps.\newline
First of all, we denote $\alpha(E)=\alpha(s,n),$ where $\alpha (E)$ is the characteristic constant of the spherical cap $E\subset \partial B_1(0),$ $s=\frac{\mathcal{H}^{n-1}(E)}{\mathcal{H}^{n-1}(\partial B_1(0))},$ and $n$ is the dimension. At this point, Theorem $2$ in \cite{FH} tells us that $\alpha(s,n)$ is a monotone decreasing function of $n$ for fixed $s,$ so 
\begin{equation}\label{definition-alpha-s-infinity}
\alpha(s,\infty)=\lim_{n\to \infty}\alpha(s,n)
\end{equation}
is well defined and satisfies $\alpha(s,\infty)\leq \alpha(s,n)$ for every $n.$ It is thus sufficient to show that $\alpha(s,\infty)\geq \psi(s)$ defined in \eqref{definition-psi}. To this end, Theorem F in \cite{FH}, which is taken by \cite{HO}, says that $\alpha(s)\geq \psi(s),$ where
\[s:=\int\limits_h^{\infty}e^{-(1/2)t^2} dt,\]
with $h=h(\alpha)$ the largest real zero of 
\[F(x)=e^{-(1/4)x^2}H_{\alpha}\left(\frac{x}{\sqrt{2}}\right)\]
satisfying
\[\frac{d^2 F}{d x^2}+\left(\alpha+\frac{1}{2}-\frac{1}{4}x^2\right)F=0\]
and
\[\frac{F'(0)}{F(0)}=-2^{1/2}\frac{\Gamma\left(\frac{1-\alpha}{2}\right)}{\Gamma\left(-\frac{\alpha}{2}\right)},\]
where $\Gamma$ is the Euler gamma function. In particular, $H_{\alpha}(x)$ is the Hermite's function of order $\alpha.$\newline
Now, Theorem $3$ in \cite{FH} shows that $\alpha(s,\infty)$ defined in \eqref{definition-alpha-s-infinity} is equal to $\alpha(s)$ of Theorem F, since $s$ of $\alpha(s,n)$ converges to $s$ of $\alpha(s)$ as $n$ goes to $\infty,$ i.e.
\[\frac{\mathcal{H}^{n-1}(E)}{\mathcal{H}^{n-1}(\partial B_1(0))}\stackrel{n\to \infty}{\to}\int\limits_h^{\infty}e^{-(1/2)t^2} dt.\]
Hence, being $\alpha(s,n)\geq \alpha(s,\infty)$ for all $n,$ we finally have that $\alpha(s,n)\geq \psi(s)$ for every $n$ and for all $s\in (0,1).$\newline
As a consequence, recalling that $s_i=\frac{\mathcal{H}^{n-1}(\Gamma_i^*)}{\mathcal{H}^{n-1}(\partial B_1(0))},$ $i\in\left\{1,2\right\},$ $\frac{s_1+s_2}{2}\leq\frac{1}{2},$ because $\Gamma_1^*\cap \Gamma_2^*=\emptyset,$ hence, since $\psi(s)$ defined in \eqref{definition-psi} is convex and decreasing, we get
\[\alpha_1+\alpha_2\geq \psi(s_1)+\psi(s_2)\geq 2\left(\frac{1}{2}\psi(s_1)+\frac{1}{2}\psi(s_2)\right)\geq 2\psi\left(\frac{s_1+s_2}{2}\right)\geq 2\psi\left(\frac{1}{2}\right)=2,\]
which finally gives \eqref{belowestimate}.\newline
An alternative proof of this result is given in \cite{CS}, where, using \cite{BKP} and \cite{BL}, the two authors directly show that $\alpha(s_1)+\alpha(s_2)\geq 2,$ exploiting the properties of $\alpha(s)$ of Theorem F in \cite{FH}, which is the first Dirichlet eigenvalue on $[h,\infty)$ associated to the Hermite operator
\[-\frac{d^2}{d x^2}+\left(\frac{1}{4}x^2-\frac{1}{2}\right).\]

\section{The main notation in the Heisenberg group}\label{secondsection}
We denote by
$\mathbb{H}^n$ the set $\mathbb{R}^{2n+1},$ $n\in \mathbb{N},$ $n\geq 1,$ 
%$(x,y,t)\in \mathbb{R}^{2n+1},$ $x,y\in \mathbb{R}^n$ 
endowed with the non-commutative inner law in such  a way that for every $P\equiv (x_1,y_1,t_1)\in \mathbb{R}^{2n+1},$ $M\equiv (x_2,y_2,t_2)\in \mathbb{R}^{2n+1},$ $x_{i}\in \mathbb{R}^n,$ $y_{i}\in \mathbb{R}^n,$ $i=1,2:$ 
$$
P\circ M:=(x_1+x_2,y_1+y_2,t_1+t_2+2(\langle x_2, y_1\rangle- \langle x_1, y_2\rangle)),
$$
where $\langle \cdot, \cdot\rangle$ denotes the usual inner product in $\mathbb{R}^n.$ Let $X_i=(e_i,0,2y_i)$ and $Y_i=(0,e_i,-2x_i),$ $i=1,\dots,n,$ where $\{e_i\}_{1\leq i\leq n}$ is the canonical basis for $\mathbb{R}^n.$

We use the same symbol to denote the vector fields associated with the previous vectors, so that for $i=1,\dots,n,$
$$
X_i=\partial_{x_i}+2y_i\partial_t,\quad
Y_i=\partial_{y_i}-2x_i\partial_t.
$$
The commutator between the vector fields is
$$
[X_i,Y_i]=-4\partial_t,\quad i=1,\ldots,n,
$$
otherwise is $0.$ 
The intrinsic gradient of a smooth function $u$ in a point $P$ is  
$$
\nabla_{\mathbb{H}^n}u(P)=\sum_{i=1}^n(X_iu(P)X_i(P)+Y_iu(P)Y_i(P)).
$$
Now, there exists a unique metric on 
%\mathfrak
$H\mathbb{H}^n_P=\mbox{span}\{X_1(P),\dots,X_n(P),Y_1(P),\dots,Y_n(P)\}$ which makes orthonormal the set of vectors $\{X_1,\dots,X_n,Y_1,\dots,Y_n\}.$ Thus, for every $P\in \mathbb{H}^n$ and for every $U,W\in H\mathbb{H}^n_P,$ $U=\sum_{j=1}^n(\alpha_{1,j}X_{j}(P)+\beta_{1,j}Y_j(P)),$
$V=\sum_{j=1}^n(\alpha_{2,j}X_{j}(P)+\beta_{2,j}Y_j(P)),$ we have
$$
\langle U,V\rangle=\sum_{j=1}^n(\alpha_{1,j}\alpha_{2,j}+\beta_{1,j}\beta_{2,j}).
$$  
In particular, we get a norm associated with the metric on the space $\mbox{span}\{X_1,\dots,X_n,Y_1,\dots,Y_n\},$ which is
$$
\mid U\mid=\sqrt{\sum_{j=1}^n\left(\alpha_{1,j}^2+\beta_{1,j}^2\right)}.
$$
For example, the norm of the intrinsic gradient of a smooth function $u$ in $P$ is
$$
\mid \nabla_{\mathbb{H}^n} u(P)\mid=\sqrt{\sum_{i=1}^n\left((X_iu(P))^2+(Y_iu(P))^2\right)}.
$$
Moreover, if $\nabla_{\mathbb{H}^n} u(P)\not=0,$ then
$$
\left|\frac{\nabla_{\mathbb{H}^n}u(P)}{\mid \nabla_{\mathbb{H}^n}u(P)\mid}\right|=1.
$$

 If $\nabla_{\mathbb{H}^n} u(P)=0,$ instead, we say that the point $P$ is characteristic for the smooth surface $\{u=u(P)\}.$
Hence, for every point $M\in \{u=u(P)\},$ which is not characteristic, it is well defined the intrinsic normal to the surface $\{u=u(P)\}$ as follows: 
$$
\nu(M)=\frac{\nabla_{\mathbb{H}^n} u(M)}{\mid \nabla_{\mathbb{H}^n} u(M)\mid}.
$$
At this point, we introduce in the Heisenberg group $\mathbb{H}^n$ the following gauge norm:
$$
| (x,y,t)|_{\mathbb{H}^n}:=\sqrt[4]{(\mid x\mid^2+\mid y\mid^2)^2+t^2}.
$$
In particular, for every positive number $r,$ the gauge ball of radius $r$ centered in $0$ is
$$
B^{\mathbb{H}^n}_r(0):=\{P\in \mathbb{H}^n :\:\:|P|_{\mathbb{H}^n}<r\}.
$$
In the Heisenberg group, a dilation semigroup is defined as follows: for every $r>0$ and for every $P=(x,y,t)\in \mathbb{H}^n,$ let
$$
\delta_r(P):=(rx,ry,r^2t).
$$

Let $P:=(\xi,\eta,\sigma)\in\mathbb{H}^n$  and $O=(0,0,0),$ then we define
 $$d_K(P,O):=|P|_{\mathbb{H}^n}.$$ 

For every $P, T\in \mathbb{H}^n$  is well defined 
$$
d_K(P,T)=|P^{-1}\circ T|_{\mathbb{H}^n},
$$ 
that is a distance $d_K$ on the Heisenberg group $\mathbb{H}^n,$ known as the Koranyi distance. This distance is left invariant, that is for every $P,T, R\in \mathbb{H}^n$
$$
d_K(R\circ P, R\circ T)=d_K(P,T).
$$ 

As a consequence, we may perform our computation supposing of dealing with $d_K(P,O)=|P|_{\mathbb{H}^n},$ where $O=(0,0,0),$ simply by multiplying the left hand side by $T^{-1}.$

 In particular, for every $i=1,\dots,n$ we obtain:
\begin{equation*}\begin{split}
X_i|P|_{\mathbb{H}^n}
%&=\frac{1}{4}((\mid \xi\mid^2+\mid \eta\mid^2)^2+\sigma^2)^{-\frac{3}{4}}(4(\mid \xi\mid^2+\mid \eta\mid^2)\xi_i+4\sigma\eta_i)=\frac{1}{4}\|P\|_{G}^{-3}(4(\mid \xi\mid^2+\mid \eta\mid^2)\xi_i+4\sigma\eta_i)\\
%&
=|P|_{\mathbb{H}^n}^{-3}((\mid \xi\mid^2+\mid \eta\mid^2)\xi_i+\sigma\eta_i)
\end{split}
\end{equation*}
and
$$
Y_i|P|_{\mathbb{H}^n}
%=\frac{1}{4}((\mid \xi\mid^2+\mid \eta\mid^2)^2+t^2)^{-\frac{3}{4}}(4(\mid \xi\mid^2+\mid \eta\mid^2)\eta_i-4\sigma\xi_i)
= |P|_{\mathbb{H}^n}^{-3}( (\mid \xi\mid^2+\mid \eta\mid^2)\eta_i- \sigma\xi_i).
$$
Moreover, for every $i=1,\dots,n:$
$$
X_i^2|P|_{\mathbb{H}^n}=-3|P|_{\mathbb{H}^n}^{-7}((\mid \xi\mid^2+\mid \eta\mid^2)\xi_i+\sigma\eta_i)^2+|P|_{\mathbb{H}^n}^{-3}(2\xi_i^2+(\mid \xi\mid^2+\mid \eta\mid^2)+2\eta_i^2)
$$
and
$$
Y_i^2|P|_{\mathbb{H}^n}=-3|P|_{\mathbb{H}^n}^{-7}((\mid \xi\mid^2+\mid \eta\mid^2)\eta_i-\sigma\xi_i)^2+|P|_{\mathbb{H}^n}^{-3}(2\eta_i^2+(\mid \xi\mid^2+\mid \eta\mid^2)+2\xi_i^2).
$$
As a consequence,
\begin{equation}\begin{split}
&\mid\nabla_{\mathbb{H}^n}|P|_{\mathbb{H}^n}\mid^2=\sum_{i=1}^n\left((X_i|P|_{\mathbb{H}^n})^2+(Y_i|P|_{\mathbb{H}^n})^2\right)
%=\|P\|_{G}^{-6}\sum_{i=1}^n\left( (\mid \xi\mid^2+\mid \eta\mid^2)^2(\xi_i^2+\eta_i^2)+\sigma^2(\xi_i^2+\eta_i^2)\right)\\
%&
%=\|P\|_{G}^{-6}\left((\mid \xi\mid^2+\mid \eta\mid^2)^3+\sigma^2(\mid \xi\mid^2+\mid \eta\mid^2)\right)
=(\mid \xi\mid^2+\mid \eta\mid^2)|P|_{\mathbb{H}^n}^{-2},
\end{split}
\end{equation}
and
\begin{equation}\begin{split}
\Delta_{\mathbb{H}^n}|P|_{\mathbb{H}^n}
%&=-3\|P\|_{G}^{-7}\left((\mid \xi\mid^2+\mid \eta\mid^2)^3+\sigma^2(\mid \xi\mid^2+\mid \eta\mid^2)\right)+(2n+4)\|P\|_{G}^{-3}(\mid \xi\mid^2+\mid \eta\mid^2)\\
%&
=(2n+1)(\mid \xi\mid^2+\mid \eta\mid^2)|P|_{\mathbb{H}^n}^{-3}.
\end{split}
\end{equation}

Thus, for every $i=1,\dots,n,$ denoting by $Q:=2n+2$ homogeneous dimension we get: 
$$
X_i|P|_{\mathbb{H}^n}^{2-Q}=(2-Q)|P|_{\mathbb{H}^n}^{1-Q}|P|_{\mathbb{H}^n}^{-3}\left((\mid \xi\mid^2+\mid \eta\mid^2)\xi_i+\sigma\eta_i\right),
$$

$$
Y_i|P|_{\mathbb{H}^n}^{2-Q}=(2-Q)|P|_{\mathbb{H}^n}^{1-Q} |P|_{\mathbb{H}^n}^{-3}\left((\mid \xi\mid^2+\mid \eta\mid^2)\eta_i- \sigma\xi_i\right),
$$
and
\begin{equation*}\begin{split}
&\Delta_{\mathbb{H}^n}|P|_{\mathbb{H}^n}^{2-Q}
%\\
%&=(2-Q)(1-Q)\|P\|_{G}^{-Q}
%\sum_{i=1}^n\left((X_i\|P\|_{G})^2+(Y_i\|P\|_{G})^2\right)+(2-Q)\|P\|_{G}^{1-Q}\sum_{i=1}^n\left(X_i^2\|P\|_{G}+Y_i^2\|P\|_{G}\right)\\
%&=(2-Q)(1-Q)\|P\|_{G}^{-Q}
%\mid\nabla_{\mathbb{H}^n}\|P\|_{G}\mid^2+(2-Q)\|P\|_{G}^{1-Q}\Delta_{\mathbb{H}^n}\|P\|_{G}\\
%&=(2-Q)\left((1-Q)\|P\|_{G}^{-2-Q}(\mid \xi\mid^2+\mid \eta\mid^2) +\|P\|_{G}^{-2-Q}(2n+1)(\mid \xi\mid^2+\mid \eta\mid^2)\right)\\
%&
=(2-Q)|P|_{\mathbb{H}^n}^{-2-Q}(\mid \xi\mid^2+\mid \eta\mid^2)\left(1-Q + 2n+1\right)=0.\end{split}
\end{equation*}

In conclusion, $|P|_{\mathbb{H}^n}^{2-Q}$ is, up to a constant, the fundamental solution of the sublaplacian $\Delta_{\mathbb{H}^n}$  in the Heisenberg group, with the pole in the origin, and
$\Gamma(P,R)=c\left|P^{-1}\circ R\right|_{\mathbb{H}^n}^{2-Q}$ is the fundamental solution of the sublaplacian $\Delta_{\mathbb{H}^n}.$

The definition of $\mathbb{H}^n-$subharmonic function, as well as the one of $\mathbb{H}^n-$superharmonic function in a set $\Omega\subset \mathbb{H}^{n},$ can be stated, as usual, in the classical way, requiring respectively that $\Delta_{\mathbb{H}^n}u(P)\geq 0$ for every $P\in \Omega,$  for the $\mathbb{H}^n-$subharmonicity, and that $\Delta_{\mathbb{H}^n}u(P)\leq 0$ for every $P\in \Omega$ for having $\mathbb{H}^n-$superharmonicity. We refer to \cite{BLU} for further details.  

Concerning the natural Sobolev spaces to consider in the Heisenberg group $\mathbb{H}^n$, we refer to the literature, see for instance \cite{GN}. Here, we simply recall that: 
$$
\mathcal{L}^{1,2}(\Omega):=\{f\in L^{2}(\Omega): X_if,\:\:Y_if\in L^{2}(\Omega),\:\:i=1,\dots, n\}
$$ 
is a Hilbert space with respect to the norm
$$
|f|_{\mathcal{L}^{1,2}(\Omega)}=\left(\int_{\Omega}(\sum_{i}^n(X_if)^2+(Y_if)^2)+|f|^2dx\right)^{\frac{1}{2}}.
$$
Moreover
$$
 H_{\mathbb{H}^n}^1(\Omega)=\overline{C^{\infty}(\Omega)\cap \mathcal{L}^{1,2}(\Omega)}^{|\cdot |_{\mathcal{L}^{1,2}(\Omega)}}.
$$

Now, if $E\subset\mathbb{H}^n$ is a measurable set, a notion of $\mathbb{H}^n$-perimeter measure
$|\partial E|_{\mathbb{H}^n}$ has been introduced in \cite{GN} in a more general setting, even if here we recall some results in the framework of the Heisenberg group, the simplest  non-trivial example of Carnot group. We refer to \cite{GN},
\cite{FSSC_houston}, \cite{FSSC_CAG}, \cite{FSSC_step2} for a detailed presentation. For
our applications, we restrict ourselves to remind that, if $E$ has locally finite $\mathbb{H}^n$-perimeter
(is a $\mathbb{H}^n$-Caccioppoli set),
then $|\partial E|_{\mathbb{H}^n}$ is a Radon measure in $\mathbb{H}^n$, invariant under
group translations and $\mathbb{H}^n$-homogeneous of degree $Q-1$. Moreover, the following
representation theorem holds (see  \cite{capdangar}).

\begin{prop}\label{perimetro regolare}
If $E$ is a $\mathbb{H}^{n}:=\mathbb{R}^{2n+1}$-Caccioppoli set with Euclidean ${\mathbf C}^1$
boundary, then there is an explicit representation of the
$\mathbb{H}^{n}$-perimeter in terms of the Euclidean $2n$-dimensional
Hausdorff measure $\mathcal H^{2n}$
\begin{equation*}
P_{\mathbb{H}^n}^{\Omega, E}(\partial E)=\int_{\partial
E\cap\Omega}\bigg(\sum_{j=1}^{n}\left(\langle
X_j,n_E\rangle_{\mathbb{R}^{2n+1}}^2+\langle
Y_j,n_E\rangle_{\mathbb{R}^{2n+1}}^2\right)\bigg)^{1/2}d{\mathcal {H}}^{2n},
\end{equation*}
where $n_E=n_E(x)$ is the Euclidean unit outward normal to $\partial
E$.
\end{prop}

We also have:
\begin{prop}\label{divergence}
If $E$ is a regular bounded open set with Euclidean ${\mathbf C}^1$
boundary and $\phi$ is a horizontal vector field,
continuously differentiable on $ \overline{\Omega} $, then
$$
\int_E \mathrm{div}_{\mathbb{H}^n}\ \phi\, dx = \int_{\partial E} \langle \phi, \nu_{\mathbb{H}^n}\rangle d P_{\mathbb{H}^n}^{E},
$$
where $\nu_{\mathbb{H}^n}(x)$ is the intrinsic horizontal unit outward normal to $\partial E$,
given by the (normalized) projection of $n_E(x)$ on the fiber $H\mathbb{H}^n_x$ of
the horizontal fiber bundle $H\mathbb{H}^n$.
\end{prop}

\begin{oss}
The definition of $\nu_{\mathbb{H}^n}$ is well done, since $H\mathbb{H}^n_x$ is transversal to
the tangent space of $E$ at $x,$ for  $P_{\mathbb{H}^n}^{E}(\partial E)$-a.e. $x\in\partial E$
(see \cite{magnani}).
\end{oss}

 Now, adapting the approach described in \cite{ACF} and recalled in Section \ref{Euclidean_setting}  to the Heisenberg case, we conclude, by applying the definition of solution in the sense of the domain variation to the functional
 \[
 \mathcal{E}_{\mathbb{H}^n}(v):=\int_{\Omega}\left(|\nabla_{\mathbb{H}^n} v|^2+\chi_{\{v>0\}}\right)dx,
\]
$\Omega \subset\mathbb{H}^n,$
 that the parallel two-phase problem to (\ref{two_phase_classical})
%  descending by the functiona $\mathcal{E}_{\mathbb{H}^n}$ 
  is, see \cite{Fe}:
\begin{equation}\label{two_phase_Heisenberg}
\begin{cases}
\Delta_{\mathbb{H}^n} u=0& \mbox{in }\Omega^+(u):= \{x\in \Omega:\hspace{0.1cm} u(x)>0\},\\
\Delta_{\mathbb{H}^n} u=0& \mbox{in }\Omega^-(u):=\mbox{Int}(\{x\in \Omega:\hspace{0.1cm} u(x)\leq 0\}),\\
|\nabla_{\mathbb{H}^n} u^+|^2-|\nabla_{\mathbb{H}^n} u^-|^2=1&\mbox{on }\mathcal{F}(u):=\partial \Omega^+(u)\cap \Omega.
\end{cases}
\end{equation}
Thus, it seems natural to consider, as a candidate for an Alt-Caffarelli-Friedman monotonicity formula in the Heisenberg group, the following function:
\begin{equation}\label{monofondformula}
J_{\beta,\mathbb{H}^n}(r)= r^{-\beta}\int_{B_r^{{\mathbb{H}^n}}(0)}\frac{\mid\nabla_{\mathbb{H}^n} u^+\mid^2}{|\zeta |_{\mathbb{H}^n}^{Q-2}}d\zeta\int_{B_r^{{\mathbb{H}^n}}(0)}\frac{\mid\nabla_{\mathbb{H}^n} u^-\mid^2}{|\zeta|_{\mathbb{H}^n}^{Q-2}}d\zeta,
\end{equation}
where $\beta>0$ is a suitable fixed exponent and $u^+:=\sup\{u,0\}$ and $u^-:=\sup\{-u,0\},$ being $0\in \mathcal{F}(u).$

 \section{Few computations in the Heisenberg group}\label{basictoolsinHeisenberg}
 
 In this section we mainly discuss some results proved in \cite{FeFo}.
\begin{lem}\label{limitato}
There exists a positive constant $c=c(Q)$ such that for every nonnegative $\mathbb{H}^n-$subharmonic function in $C(B_1^{\mathbb{H}^n}(0)),$ if $u(0)=0,$ then there exists $r_0$ such that for every $0<\rho<r_0:$
$$
\int_{B_\rho^{\mathbb{H}^n}(0)}\frac{\mid\nabla_{\mathbb{H}^n} u(\zeta)\mid^2}{|\zeta|_{\mathbb{H}^n}^{Q-2}}d\zeta\leq  c\rho^{-Q}\int_{B_{2\rho}^{\mathbb{H}^n}(0)\setminus B_{\rho}^{\mathbb{H}^n}(0)}u^2(\zeta)d\zeta.
$$
\end{lem}
 
 \begin{lem}\label{lemmaapplicant}
For every nonnegative $\mathbb{H}^n-$subharmonic functions  $u_i\in C(B_1^{\mathbb{H}^n}(0)),$ $i=1,2,$  such that $u_1u_2=0$  and  $u_1(0)=u_2(0)=0,$  we have
%the function $J:(0,1)\to \mathbb{R}^+$ is such that 
 $$
 \frac{J_{\beta,\mathbb{H}^n}'(1)}{J_{\beta,\mathbb{H}^n}(1)}=\frac{\int_{\partial B_1^{{\mathbb{H}^n}}(0)}\frac{\mid\nabla_{\mathbb{H}^n} u_1(\kappa)\mid^2}{\sqrt{\mid x\mid^2+\mid y\mid^2}}dP^{B_1^{{\mathbb{H}^n}}(0)}_{\mathbb{H}^n}(\kappa)}{\int_{B_1^{{\mathbb{H}^n}}(0)}\frac{\mid\nabla_{\mathbb{H}^n} u_1(\kappa)\mid^2}{|\kappa |_{\mathbb{H}^n}^{Q-2}}d\kappa}+\frac{\int_{\partial B_1^{{\mathbb{H}^n}}(0)}\frac{\mid\nabla_{\mathbb{H}^n} u_2(\kappa)\mid^2}{\sqrt{\mid x\mid^2+\mid y\mid^2}}dP^{B_1^{{\mathbb{H}^n}}(0)}_{\mathbb{H}^n}(\kappa)}{\int_{B_1^{{\mathbb{H}^n}}(0)}\frac{\mid\nabla_{\mathbb{H}^n} u_2(\kappa)\mid^2}{|\kappa |_{\mathbb{H}^n}^{Q-2}}d\kappa}-\beta.
 $$
Moreover, $J_{\beta,\mathbb{H}^n}$ will be monotone increasing in the interval $[0,r_0),$ for some $r_0>0,$ if and only if $\frac{J'_{\beta,\mathbb{H}^n}(1)}{J_{\beta,\mathbb{H}^n}(1)}\geq 0$ for every $u_1,$ $u_2$ satisfying the hypotheses of this lemma.
%then there exists $r_0>0$ such that $J_{\beta,\mathbb{H}^n}$ is monotone increasing in $[0,r_0)$.
\end{lem}
In order to obtain some estimates of $\frac{J'_{\beta,\mathbb{H}^n}(1)}{J_{\beta,\mathbb{H}^n}(1)},$ we need to read the Kohn-Laplace operator $\Delta_{\mathbb{H}^n}$ in terms of radial coordinates. The problem has been faced in \cite{Jerison}, by using an abstract and elegant approach, see also \cite{Greiner} and \cite{Biri}. In \cite{FeFo} we describe the $\mathbb{H}^1$ case in details, with an explicit computation. 

Precisely, we consider the following coordinates in $\mathbb{H}^{1}:$ 
\begin{equation}\label{pol-coord}
T(\rho,\varphi,\theta):=
\begin{cases}
x=\rho \sqrt{\sin \varphi}\cos\theta \\
y=\rho \sqrt{\sin \varphi}\sin \theta\\
t=\rho^{2}\cos \varphi.
\end{cases}
\end{equation}

From \eqref{pol-coord}, we obtain the values of $\rho,$ $\varphi$ and $\theta$ with respect to the cartesian coordinates $x,$ $y$ and $t,$ that is: 
\begin{equation}\label{cartes-coord}
\begin{cases}
\rho=((x^{2}+y^{2})^{2}+t^{2})^{1/4}\\
\theta=\arctan\left({\frac{y}{x}}\right)\\
\varphi=\arccos\left(\frac{t}{\rho^{2}}\right).
\end{cases}
\end{equation}
Recalling the vector fields
\begin{equation}\label{heisen-vect-fields}
\begin{cases}
X=\frac{\partial}{\partial x}+2y\frac{\partial}{\partial t}\\
Y=\frac{\partial}{\partial y}-2x\frac{\partial}{\partial t},
\end{cases}
\end{equation}
and the operators:
\begin{equation}\label{grad-lapl-horiz}
\nabla_{\mathbb{H}^1}\equiv(X,Y),\quad \Delta_{\mathbb{H}^1}=X^{2}+Y^{2},
\end{equation}
we determine the following: $\nabla_{\mathbb{H}^1}\rho,$ $\nabla_{\mathbb{H}^1}\theta,$ $\nabla_{\mathbb{H}^1}\varphi,$ by using \eqref{heisen-vect-fields}, \eqref{cartes-coord} and \eqref{grad-lapl-horiz}.
\begin{lem}
Let $\rho, \varphi, \theta$ defined as in (\ref{pol-coord}). Then:
\[
\nabla_{\mathbb{H}^{1}}\rho=\rho^{-3}((x^{2}+y^{2})x+ty,(x^{2}+y^{2})y-tx),\quad \nabla_{\mathbb{H}^{1}}\varphi=\frac{2}{\rho(x^{2}+y^{2})}\left(t\nabla_{\mathbb{H}^{1}}\rho+\rho(-y,x)\right)
\]
and 
$$
\nabla_{\mathbb{H}^{1}}\theta=\frac{1}{x^{2}+y^{2}}(-y,x).
$$
\end{lem}
In addition, we obtain the properties described in the following lemma.
\begin{lem}\label{raccolta}
Let $\rho, \varphi, \theta$ defined as in (\ref{pol-coord}). Then:
\begin{align*}
%\label{norm-quad-grad-horiz-varphi-rho-theta}
\left|\nabla_{\mathbb{H}^{1}}\varphi\right|^{2}=\frac{4(x^{2}+y^{2})}{\rho^{4}},\nonumber\quad
\left|\nabla_{\mathbb{H}^{1}}\rho\right|^{2}=\frac{x^{2}+y^{2}}{\rho^{2}},\nonumber\quad
\left|\nabla_{\mathbb{H}^{1}}\theta\right|^{2}=\frac{1}{x^{2}+y^{2}}.
\end{align*}
Moreover, it results:
$$\langle\nabla_{\mathbb{H}^{1}}\varphi, \nabla_{\mathbb{H}^{1}}\rho\rangle=0\nonumber,\quad \langle\nabla_{\mathbb{H}^{1}}\rho, \nabla_{\mathbb{H}^{1}}\theta\rangle=-\frac{\cos\varphi}{\rho},\quad \langle\nabla_{\mathbb{H}^{1}}\varphi,\nabla_{\mathbb{H}^{1}}\theta\rangle=\frac{2(x^{2}+y^{2})}{\rho^{4}}$$
and
$$
\Delta_{\mathbb{H}^1}\theta=0, \quad \Delta_{\mathbb{H}^1}\rho=\frac{3(x^{2}+y^{2})}{\rho^{3}},\quad \Delta_{\mathbb{H}^1}\varphi=\frac{4\cos \varphi}{\rho^{2}}.
$$
\end{lem}

 Let now $\nabla_{\mathbb{H}^1}u(P)\in H\mathbb{H}^1_P$ and  define
 $$
 e_\rho:=\frac{\nabla_{\mathbb{H}^1}\rho}{|\nabla_{\mathbb{H}^1}\rho|},\quad \mbox{and}\quad  e_\varphi:=\frac{\nabla_{\mathbb{H}^1}\varphi}{|\nabla_{\mathbb{H}^1}\varphi|}.
 $$
 We observe that $\langle e_\rho,e_\varphi\rangle_{\mathbb{R}^2}=0,$ see Lemma \ref{raccolta}.  Then, whenever $e_\rho,e_\varphi$ exist  we have:
 $$
 \mbox{span}\{e_\rho(P),e_\varphi(P)\}= H\mathbb{H}^1_P.
 $$ 
 As a consequence,
 $$
 \nabla_{\mathbb{H}^1}u(P)=\langle \nabla_{\mathbb{H}^1}u(P),e_\rho(P)\rangle e_\rho(P)+\langle \nabla_{\mathbb{H}^1}u(P),e_\varphi(P)\rangle e_\varphi(P)
 $$
 and denoting $\nabla^\rho_{\mathbb{H}^1}u(P)=\langle \nabla_{\mathbb{H}^1}u(P),e_\rho(P)\rangle e_\rho(P)$ and
 $\nabla^\varphi_{\mathbb{H}^1}u(P)=\langle \nabla_{\mathbb{H}^1}u(P),e_\varphi(P)\rangle e_\varphi(P),$ we have
 $$
 |\nabla_{\mathbb{H}^1}u(P)|^2=\langle \nabla_{\mathbb{H}^1}u(P),e_\rho(P)\rangle^2+\langle \nabla_{\mathbb{H}^1}u(P),e_\varphi(P)\rangle^2.
 $$
 and
\begin{equation}\label{norm-quad-lapl-horiz-u-two-components_f}
 |\nabla_{\mathbb{H}^1}u(P)|^2= |\nabla^\rho_{\mathbb{H}^1}u(P)|^2+ |\nabla^\varphi_{\mathbb{H}^1}u(P)|^2.
\end{equation}
We may summarize this fact as follows.
\begin{lem}
The couple
 $(\nabla_{\mathbb{H}^1}\rho)(P)$ , $(\nabla_{\mathbb{H}^1}\varphi)(P)$ determines a basis of $H\mathbb{H}^1_P,$ for every $P=(x,y,t),$ such that $x^{2}+y^{2}\neq 0.$
\end{lem}
At this point, assuming that $u=\rho^{\alpha}f(\theta,\varphi),$  we compute $\Delta_{\mathbb{H}^1}u$ obtaining the following result.
\begin{lem}\label{lapla}
Let $u=\rho^{\alpha}f(\theta,\varphi),$  then:
\begin{equation*}
\begin{split}
%\label{lapl-horiz-u-final-2}
&\Delta_{\mathbb{H}^1}u=\Delta_{\mathbb{H}^1}(\rho^{\alpha}f(\theta,\varphi))=\rho^{\alpha-2}\bigg(\alpha(\alpha+2)(\sin\varphi)f(\theta,\varphi)-2\alpha(\cos\varphi)\frac{\partial f}{\partial \theta}\nonumber\\
&+\frac{1}{\sin\varphi}\frac{\partial^{2}f}{\partial \theta^{2}}+4\sin\varphi\frac{\partial^{2}f}{\partial \varphi \partial \theta}+4\sin\varphi\frac{\partial^{2}f}{\partial \varphi^{2}}+4\cos\varphi\frac{\partial f}{\partial \varphi}\bigg).
\end{split}
\end{equation*}
In particular, if $u=\rho^{\alpha}f(\varphi),$ then
\begin{equation*} \label{lapl-horiz-varphi-boundary-boule-koranyi}
\Delta_{\mathbb{H}^1}u=\rho^{\alpha-2}\bigg(\alpha(\alpha+2)(\sin \varphi)f(\varphi)+4\frac{\partial}{\partial \varphi}\left(\sin\varphi\frac{\partial f}{\partial \varphi}\right)\bigg).
\end{equation*}
\end{lem}
Thus, whenever $f$ satisfies 
$$
\alpha(\alpha+2)(\sin\varphi)f(\theta,\varphi)-2\alpha(\cos\varphi)\frac{\partial f}{\partial \theta}+\frac{1}{\sin\varphi}\frac{\partial^{2}f}{\partial \theta^{2}}+4\sin\varphi\frac{\partial^{2}f}{\partial \varphi \partial \theta}+4\sin\varphi\frac{\partial^{2}f}{\partial \varphi^{2}}+4\cos\varphi\frac{\partial f}{\partial \varphi}=0
$$
on $ \Gamma\subset \partial B_1^{{\mathbb{H}^1}}(0)$ or 
$$
\alpha(\alpha+2)(\sin \varphi)f(\varphi)+4\frac{\partial}{\partial \varphi}\left(\sin\varphi\frac{\partial f}{\partial \varphi}\right)=0
$$
on $ \Gamma\subset \partial B_1^{{\mathbb{H}^1}}(0)$
for $f$ depending only on $\varphi,$ then $u=\rho^\alpha f(\theta,\varphi)$ is $\mathbb{H}^1$-harmonic in the set
$$
\mathcal{P}_\Gamma:=\{(x,y,t)\in \mathbb{H}^1:\quad (x,y,t)=\delta_{\lambda} (\xi,\eta,\tau),\:\:\lambda>0,\:\: (\xi,\eta,\tau)\in \Gamma \subset \partial B_1^{{\mathbb{H}^1}}(0) \}.
$$
In fact, if $u_\lambda(x,y,t)=u(\delta_\lambda (x,y,t)),$ then whenever $u=\rho^\alpha f(\theta, \phi),$ $u_\lambda =\lambda^\alpha u(x,y,t)$ and if $u$ is $\mathbb{H}^1-$harmonic on $ \Gamma\subset \partial B_1^{{\mathbb{H}^1}}(0),$ we obtain:
\begin{equation}
\begin{split}&\Delta_{\mathbb{H}^1}u_{\lambda}(x,y,t)=\lambda^\alpha\Delta_{\mathbb{H}^1}u(x,y,t)=0.
%\\
%&=\lambda^{4-\alpha}(\sqrt[4]{(x^2+y^2)^2+t^2})^{2-\alpha}\Delta_{\mathbb{H}^1}( f(\arctan\left({\frac{y}{x}}\right),\arccos\left(\frac{t}{\sqrt{(x^2+y^2)^2+t^2}}\right)))\\
%&=
\end{split}
\end{equation}
For instance, if $\Gamma=\{(x,y,t)\in\partial B_1^{{\mathbb{H}^1}}(0):\quad x^2+y^2<Mt\},$ where $M>0$ is a constant, then

$$
\mathcal{P}_\Gamma=\{(x,y,t)\in \mathbb{H}^1:\quad x^2+y^2<Mt \}.
$$
Moreover, if we add a boundary condition  to the equation 
$$
\alpha(\alpha+2)(\sin\varphi)f(\theta,\varphi)-2\alpha(\cos\varphi)\frac{\partial f}{\partial \theta}+\frac{1}{\sin\varphi}\frac{\partial^{2}f}{\partial \theta^{2}}+4\sin\varphi\frac{\partial^{2}f}{\partial \varphi \partial \theta}+4\sin\varphi\frac{\partial^{2}f}{\partial \varphi^{2}}+4\cos\varphi\frac{\partial f}{\partial \varphi}=0,
$$
by requiring that $f=0$ on $\partial \Gamma,$
then $u=\rho^\alpha f$ satisfies 
$$\left\{\begin{array}{ll}\Delta_{\mathbb{H}^1}u=0,&(x,y,t)\in \mathcal{P}_\Gamma,\\
u=0,&(x,y,t)\in \partial \mathcal{P}_\Gamma.\end{array}\right.$$ 

Of course, if we fix $\Gamma$ and e assume that $f=0$ on $\partial \Gamma$ as well, then the equation  
$$
\alpha(\alpha+2)(\sin\varphi)f(\theta,\varphi)-2\alpha(\cos\varphi)\frac{\partial f}{\partial \theta}+\frac{1}{\sin\varphi}\frac{\partial^{2}f}{\partial \theta^{2}}+4\sin\varphi\frac{\partial^{2}f}{\partial \varphi \partial \theta}+4\sin\varphi\frac{\partial^{2}f}{\partial \varphi^{2}}+4\cos\varphi\frac{\partial f}{\partial \varphi}=0
$$
has a solution only for some particular values of $\alpha.$ Indeed, focusing our attention to the case in which $f$ depends only on $\varphi,$ we clearly obtain the following eigenvalues problem:
$$\left\{\begin{array}{ll}4\left(\sin\varphi f'\right)'=-\lambda_{\varphi_0,\varphi_1}\sin \varphi f,&0\leq \varphi_0<\varphi<\varphi_1\leq \pi,\\
f(\varphi_0)=0=f(\varphi_1),&
\end{array}\right.$$ 
where $\Gamma:=\{(x,y,t)\in \partial B_1^{{\mathbb{H}^1}}(0):\quad \cos\varphi_1<\frac{t}{\sqrt{(x^2+y^2)^2+t^2}}<\cos\varphi_0 \}.$ Hence, the exponent $\alpha$ is related to the first eigenvalue $\lambda_{\varphi_0,\varphi_1}$ via the relationship:
$$\lambda_{\varphi_0,\varphi_1}=\alpha(\alpha+2).$$

On the other hand, the first eigenvalue $\lambda_{\varphi_0,\varphi_1}$ is determined by the Rayleigh quotient given, in this case, by
$$
\lambda_{\varphi_0,\varphi_1}:=\inf_{f\hspace{0.025cm}\in\hspace{0.025cm} H_0^1(\varphi_0,\varphi_1)}\frac{\displaystyle4\int_{\varphi_0}^{\varphi_1}\sin(\varphi)f'(\varphi)^2d\varphi}{\displaystyle\int_{\varphi_0}^{\varphi_1}\sin(\varphi)f(\varphi)^2d\varphi}.
$$

Thus, it is fundamental to know if the result by \cite{FH}, that is the cap on $\partial B_1(0)$ having the same $\mathcal{H}^{n-1}$ measure of some sets $\Sigma$ on $\partial B_1(0)$ has the smallest Rayleigh quotient, is true even in the Heisenberg case. 

Let say that we would like to know if there exists a set $\Gamma^*\subset \partial B_1^{{\mathbb{H}^1}}(0)$ such that for every
$\Gamma\subset \partial B_1^{{\mathbb{H}^1}}(0),$
$$
P_{\mathbb{H}^1}^{B_1^{{\mathbb{H}^1}}(0)}(\Gamma)= P_{\mathbb{H}^1}^{B_1^{{\mathbb{H}^1}}(0)}(\Gamma^*),
$$
it results
$$
\alpha_{\mathbb{H}^1}(\Gamma)\geq \alpha_{\mathbb{H}^1}(\Gamma^*),
$$
where $\alpha_{\mathbb{H}^1}(\Gamma)$ denotes the unique positive solution to the equation
$$
\alpha(\alpha+2)=\lambda(\Gamma),
$$
 $\lambda(\Gamma)$ is the first eigenvalue of the problem
$$\left\{\begin{array}{ll}\mathcal{L}_{\theta,\varphi}f=-\lambda (\Gamma)f& \mbox{in }\Omega\subset \mathbb{R}^2,\\
f=0& \mbox{on }\partial\Omega,\end{array}\right.$$
with
\[ \mathcal{L}_{\theta,\varphi}=\diver_{\theta,\varphi}\left(A(\theta,\varphi)\nabla_{\theta,\varphi}\right)\]
where $T(\Omega)=\Gamma$ and $A(\theta,\varphi)$ is the matrix-valued function 
\begin{equation}\label{matrix-divergence-form}
A(\theta,\varphi)=
\begin{bmatrix}
\displaystyle\frac{1}{\sin\varphi}&(4+2\alpha)\sin\varphi\\
\\
-2\alpha\sin\varphi& 4\sin\varphi
\end{bmatrix}.
\end{equation}

In particular, it holds:
\begin{equation}\label{weak-equation-theta-phi-6}
\lambda (\Gamma)=\inf_{v\hspace{0.025cm}\in\hspace{0.025cm} H_0^1(\Omega_{\theta,\varphi})}\frac{\displaystyle\int\limits_{\Omega_{\theta,\varphi}}\left(\frac{1}{\sin\varphi}\left(\frac{\partial f}{\partial \theta}\right)^{2}+4\sin\varphi\frac{\partial f}{\partial \theta}\frac{\partial f}{\partial \varphi}+4\sin\varphi\left(\frac{\partial f}{\partial \varphi}\right)^{2}\right)\hspace{0.1cm}d\theta d\varphi}{\displaystyle\int\limits_{\Omega_{\theta,\varphi}}(\sin\varphi)f^{2}\hspace{0.1cm}d\theta d\varphi}.
\end{equation}
The existence in the Heisenberg group of the properties of the characteristic number associated with the set $\Gamma$, as far as we know, is still unknown. This part corresponds to the topic discussed in \cite{Sperner} in the Euclidean setting.
In fact, just for having an idea about the difficulty in solving the problem, we remark that  
$$
P_{\mathbb{H}^1}^{B_1^{{\mathbb{H}^1}}(0)}(\Gamma)=\int_{\Omega}\sqrt{\sin(\varphi)}d\theta d\varphi,
$$
where $\Gamma=T(\{1\}\times\Omega).$ At this point, we may decide to symmetrize the set $\Omega$ in many ways. For instance, for every $\varphi,$ we might define $\Omega^*_{\varphi}$ in such a way that
$$
\mathcal{H}^{1}(\Omega^*_\varphi)=2\theta_\varphi=\mathcal{H}^{1}(\Omega_\varphi),
$$ 
and consider $\Omega^*:=\cup_{\varphi\in\Pi_2(\Omega)} \Omega_\varphi^*,$ being $\Pi_2(\Omega):=\{\varphi:\:\:\Omega_\varphi\not=\emptyset\}.$
Unfortunately, the lack of an isoperimetric result does not permit to conclude anything.

What we can do is to give an estimate. In fact, let
\begin{equation}\label{def-lambda-varphixx}
\lambda_{\varphi}(\Sigma):=\inf_{v\hspace{0.025cm}\in\hspace{0.025cm}H^{1}_{0}(\Sigma)}\frac{\displaystyle \int\limits_{\Sigma}\frac{\left|\nabla^{\varphi}_{\mathbb{H}^1}v(\xi)\right|^{2}}{\sqrt{x^{2}+y^{2}}}\hspace{0.1cm}dP^{B_1^{{\mathbb{H}^1}}(0)}_{\mathbb{H}^1}(\xi)}{\displaystyle \int\limits_{\Sigma}v^{2}(\xi)\sqrt{x^{2}+y^{2}}\hspace{0.1cm}dP^{B_1^{{\mathbb{H}^1}}(0)}_{\mathbb{H}^1}(\xi)},
\end{equation}
be the Rayleigh quotient,
where $\Sigma\subset \partial B_1^{\mathbb{H}^1}(0)$ is a rectifiable set,
%Here $P_{\mathbb{H}^1}$ denotes the perimeter measure in the Heisenberg group, see \cite{CDST}, $\nabla^\varphi_{\mathbb{H}^1}v(\xi)=\langle \nabla_{\mathbb{H}^1}u(\xi),e_\varphi\rangle e_\varphi(\xi),$ 
%$e_\varphi:=\frac{\nabla_{\mathbb{H}^1}\varphi}{|\nabla_{\mathbb{H}^1}\varphi|}$ and $\varphi$ is the variable associated with the point $\xi\in \mathbb{H}^1$ via the following appropriate spherical coordinates in the Heisenberg group
then the following result holds, see \cite{FeFo}.

\begin{teo}\label{lowerbound_f}
Let $u_1,u_2\in C(B_1^{\mathbb{H}^1}(0))\cap H^1_{\mathbb{H}^1}(B_1^{\mathbb{H}^1}(0))$ be nonnegative, such that  $u_1u_2=0$ in $B_1^{\mathbb{H}^1}(0)$ and $u_i(0)=0,$ $\Delta_{\mathbb{H}^1} u_i\geq 0,$ $i=1,2.$ Then
\begin{equation}\label{general-term-condition-monotonicity-lower-bound}
\sum_{i=1}^2\frac{\displaystyle \int\limits_{\partial B_1^{\mathbb{H}^1}(0)}\frac{\left|\nabla_{\mathbb{H}^1}u_i(\xi)\right|^{2}}{\sqrt{x^{2}+y^{2}}}\hspace{0.1cm}d P^{B_1^{{\mathbb{H}^1}}(0)}_{\mathbb{H}^1}(\xi)}{\displaystyle \int\limits_{B_1^{\mathbb{H}^1}(0)}\frac{\left|\nabla_{\mathbb{H}^1}u_i(\xi)\right|^{2}}{\left|\xi\right|_{\mathbb{H}^1}^{2}}\hspace{0.1cm}d\xi}\geq 2\sum_{i=1}^2\left(\sqrt{1+\lambda_{\varphi}\left(\Sigma_i\right)}-1\right),
\end{equation}
where $\Sigma_i=\partial B_1^{\mathbb{H}^1}(0)\cap \left\{u_i>0\right\}.$
\end{teo}

In fact, recalling the particular structure described in  (\ref{norm-quad-lapl-horiz-u-two-components_f}), we have:
\begin{equation*}
 |\nabla_{\mathbb{H}^1}u|^2= |\nabla^\rho_{\mathbb{H}^1}u|^2+ |\nabla^\varphi_{\mathbb{H}^1}u|^2. 
\end{equation*}
Thus, as well as in the Euclidean setting, we obtain
the following lower bound for each function $u:=u_i,$ $i=1,2$:
\[ \frac{\displaystyle \int\limits_{\partial B_1^{\mathbb{H}^1}(0)}\frac{\left|\nabla_{\mathbb{H}^1}u(\xi)\right|^{2}}{\sqrt{x^{2}+y^{2}}}\hspace{0.1cm}dP_{\mathbb{H}^1}(\xi)}{\displaystyle \int\limits_{B_1^{\mathbb{H}^1}(0)}\frac{\left|\nabla_{\mathbb{H}^1}u(\xi)\right|^{2}}{\left|\xi\right|_{\mathbb{H}^1}^{2}}\hspace{0.1cm}d\xi}\geq \frac{A_{\rho}+A_{\varphi}}{A_{u}+A_{u}^{1/2}A_{\rho}^{1/2}},\]
where
\begin{equation}\label{def-A-rho-A-varphi-A-u_x}
\begin{split}
A_{\rho}:=\int\limits_{\partial B_1^{\mathbb{H}^1}(0)}\frac{\left|\nabla_{\mathbb{H}^1}^\rho u(\xi)\right|^{2}}{\sqrt{x^{2}+y^{2}}}&\hspace{0.1cm}dP_{\mathbb{H}^1}(\xi),\quad
A_{\varphi}:=\int\limits_{\partial B_1^{\mathbb{H}^1}(0)}\frac{\left|\nabla_{\mathbb{H}^1}^\varphi u(\xi)\right|^{2}}{\sqrt{x^{2}+y^{2}}}\hspace{0.1cm}dP_{\mathbb{H}^1}(\xi),\\
&A_{u}:=\int\limits_{\partial B_1^{\mathbb{H}^1}(0)}u^{2}(\xi)\sqrt{x^{2}+y^{2}}\hspace{0.1cm}dP_{\mathbb{H}^1}(\xi),
\end{split}
\end{equation}

having denoted the perimeter measure $P^{B_1^{{\mathbb{H}^1}}(0)}_{\mathbb{H}^1}$ simply by $P_{\mathbb{H}^1}$ and being $\xi=(x,y,t)\in \mathbb{H}^1.$
Hence, recalling the definition (\ref{def-lambda-varphixx}) we conclude. In addition, it results that 
$\sqrt{1+\lambda_{\varphi}\left(\Sigma_i\right)}-1$ is the positive solution of $\alpha^2+2\alpha-\lambda_{\varphi}\left(\Sigma_i\right)=0,$
  see \cite{FeFo} for the details.
\section{Friedland-Hayman result and the existence of a monotonicity formula in Heisenberg group }\label{lackofisoperimetric}
In \cite{FH}, it is proved that for every $\Sigma\subset \partial B_1(0)$ such that $\mathcal{H}^{m-1}(\Sigma)=\mathcal{H}^{m-1}(\Sigma^*_{\varphi_0}),$ where $\Sigma^*_{\varphi_0}$ is a cap with width $\varphi_0,$ if  $f\in H^1_0(\partial B_1(0))$  and $B_1(0)\subset \mathbb{R}^m$ is the Euclidean ball of radius one centered at $0,$ then 
$$
\frac{\displaystyle\int_{\Sigma}\mid\nabla f\mid^2d\sigma}{\displaystyle\int_{\Sigma}f^2d\sigma}\geq H(f^*)=\frac{\displaystyle\int_{\Sigma^*_{\varphi_0}}\mid\nabla f^*\mid^2d\sigma}{\displaystyle\int_{\Sigma^*_{\varphi_0}}(f^*)^2d\sigma}=\frac{\displaystyle\int_{0}^{\varphi_0}(f^*)'(\varphi)^2\sin^{m-2}(\varphi)d\varphi}{\displaystyle\int_{0}^{\varphi_0}f^*(\varphi)^2\sin^{m-2}(\varphi)d\varphi}\geq H(F)=\lambda_E(\Sigma^*_{\varphi_0}) 
$$
with $F$ the solution of the eigenvalues problem
$$
\left\{\begin{array}{ll}
F''+(m-2)\cot(\varphi)F'+\lambda_E(\Sigma^*_{\varphi_0}) F=0,&\varphi\in ]0,\varphi_0[,\\
F(\varphi_0)=0,&\\
F'(0)=0,&
\end{array}
\right.
$$
where $\lambda_E(\Sigma^*_{\varphi_0})$ is the first eigenvalue and
$\varphi_0 \in [0,\pi],$ see Lemma 1 in \cite{FH}. Moreover, see Lemma 2 in \cite{FH}, the function $w=\rho^{\alpha(\Sigma^*_{\varphi_0})}F$ is harmonic in the Euclidean cone  having as a cap $\Sigma^*_{\varphi_0}$ on $\partial B_1(0),$ with opening $\varphi_0,$ and where $\alpha(\Sigma^*_{\varphi_0})$ is the characteristic number associated with $\Sigma^*_{\varphi_0},$ always in the Euclidean framework.  
In case $m=3,$ we exactly obtain
$$
\lambda_{0,\varphi_0}=4\lambda_E(\Sigma_{\varphi_0}^*),
$$ 
so that the relationship with $\alpha_{\mathbb{H}^1}$ becomes
$$
\alpha(\alpha+2)=4\lambda_E(\Sigma_{\varphi_0}^*)
$$
and
$$
\alpha_{\mathbb{H}^1}(\Sigma_{\varphi_0}^{\partial B_1^{\mathbb{H}^1}(0)})=\sqrt{1+4\lambda_E(\Sigma_{\varphi_0}^*)}-1.
$$
Thus, we can deduce that the minimum is realized when $\varphi_0=\frac{\pi}{2},$ and since $\lambda_E(\Sigma_{\frac{\pi}{2}}^*)=2,$ we conclude that $\alpha_{\mathbb{H}^1}(\Sigma_{\frac{\pi}{2}}^{\partial B_1^{\mathbb{H}^1}(0)})=2.$ Furthermore,
$$
\sqrt{1+4\lambda_E(\Sigma_{\varphi_0}^*)}-1\leq \sqrt{1+4\lambda_E(\Sigma)}-1,$$
nevertheless we can not conclude, in general, that $\sqrt{1+4\lambda_E(\Sigma)}-1\leq \sqrt{1+\lambda_{\mathbb{H}^1}(\Sigma)}-1,$ except when the cap on $\partial B_1^{\mathbb{H}^1}(0)$ depends only on $\varphi_0.$ 
In this particular case, for functions having these caps depending only on $\varphi,$ by choosing $\beta=8$ in (\ref{monofondformula}), we conclude that $\frac{J'_{\beta,\mathbb{H}^1}(1)}{J_{\beta,\mathbb{H}^1}(1)}\geq 0$ and Lemma \ref{lemmaapplicant} applies. Indeed, by straightforward  computation, by considering the function $u=\alpha t^+-\beta t^-,$ for some positive numbers $\alpha, \beta,$ and denoting
$$
J_{8,\mathbb{H}^1}^{\alpha t^+-\beta t^-}(r)=r^{-8}\int_{B_r^{{\mathbb{H}^1}}(0)}\frac{\mid\nabla_{\mathbb{H}^n} t^+\mid^2}{|\zeta |_{\mathbb{H}^1}^{2}}d\zeta\int_{B_r^{{\mathbb{H}^1}}(0)}\frac{\mid\nabla_{\mathbb{H}^1} t^-\mid^2}{|\zeta |_{\mathbb{H}^1}^{2}}d\zeta,
$$
it results that $$\frac{d}{dr}J_{8,\mathbb{H}^1}^{\alpha t^+-\beta t^-}(r)=0.$$

This is perfectly coherent with the fact that  $(0,0,0)$ is a characteristic point of the surface  $\{(x,y,t)\in B_1^{{\mathbb{H}^1}}(0):\quad t=0\},$ so that in $(0,0,0)$ it holds $\nabla_{\mathbb{H}^1}t_{\mid (0,0,0)}\equiv (0,0),$ as well as we know that
$$\frac{d}{dr}J_{8,\mathbb{H}^1}^{\alpha t^+-\beta t^-}(r)=0=\mid\nabla_{\mathbb{H}^1}\alpha t^{+}_{\mid (0,0,0)}\mid \mid\nabla_{\mathbb{H}^1}\beta t^{-}_{\mid (0,0,0)}\mid.$$

We recall that in this case:
$$\{(x,y,t)\in B_1^{{\mathbb{H}^1}}(0):\quad \alpha t^+-\beta t^-=0\}=\{(x,y,t)\in B_1^{{\mathbb{H}^1}}(0):\quad t=0\}.$$  
We observe here, as a by-product, that $\alpha t^+-\beta t^-$ can not be a classical solution of (\ref{two_phase_Heisenberg}), since the free boundary condition is not fulfilled in $(0,0,0).$
Moreover, we remark that this situation corresponds to the case in which the Koranyi ball is split in two parts separated by the plane $\{(x,y,t)\in \mathbb{H}^1,\quad t=0\}.$

On the other hand, whenever we fix $\alpha,\beta \geq 0$  such that $\alpha^2-\beta^2=1,$ the function $u=\alpha x^+-\beta x^-$ is a solution of the two-phase problem (\ref{two_phase_Heisenberg}), as well as in the Euclidean case, but in this case, testing $J_{8,\mathbb{H}^1}'$ on this function, we would get that the differential $J_{8,\mathbb{H}^1}'$ is negative, loosing in this way the desired monotonicity property of our formula, see \cite{FeFo}.

In fact, the following result holds.
\begin{lem}\label{lemma-integral-quotient-gradient-x-+}
	For every $a,b\in\mathbb{R},$ such that $a\not =0$ or $b\not=0,$  let $u=(ax+by)^+,$ defined in $B_1^{\mathbb{H}^1}(0).$  Then
	\[\frac{\displaystyle \int\limits_{\partial B_1^{\mathbb{H}^1}(0)}\frac{\left|\nabla_{\mathbb{H}^1}u\right|^2}{\sqrt{x^2+y^2}}\hspace{0.1cm}dP_{\mathbb{H}^1}(\xi)}{\displaystyle \int\limits_{ B_1^{\mathbb{H}^1}(0)}\frac{\left|\nabla_{\mathbb{H}^1}u\right|^2}{\left|\xi\right|_{\mathbb{H}^1}^2}\hspace{0.1cm}d\xi}=2.\]
\end{lem}
As a consequence of previous Lemma \ref{lemma-integral-quotient-gradient-x-+} and Lemma \ref{lemmaapplicant}, we obtain the proof of Theorem \ref{corcor1}.  
\begin{proof}[Proof of Theorem \ref{corcor1}]
From Lemma \ref{lemmaapplicant} we know that $\frac{J_{\beta,\mathbb{H}^1}'(1)}{J_{\beta,\mathbb{H}^1}(1)}\geq 0$ if and only if
\begin{equation}\label{FFF1}
 \frac{\int_{\partial B_1^{{\mathbb{H}^1}}(0)}\frac{\mid\nabla_{\mathbb{H}^1} u_1(\kappa)\mid^2}{\sqrt{x^2+y^2}}dP^{B_1^{{\mathbb{H}^1}}(0)}_{\mathbb{H}^1}(\kappa)}{\int_{B_1^{{\mathbb{H}^1}}(0)}\frac{\mid\nabla_{\mathbb{H}^1} u_1(\kappa)\mid^2}{|\kappa |_{\mathbb{H}^1}^{2}}d\kappa}+\frac{\int_{\partial B_1^{{\mathbb{H}^1}}(0)}\frac{\mid\nabla_{\mathbb{H}^1} u_2(\kappa)\mid^2}{\sqrt{ x^2+ y^2}}dP^{B_1^{{\mathbb{H}^1}}(0)}_{\mathbb{H}^1}(\kappa)}{\int_{B_1^{{\mathbb{H}^1}}(0)}\frac{\mid\nabla_{\mathbb{H}^1} u_2(\kappa)\mid^2}{|\kappa |_{\mathbb{H}^1}^{2}}d\kappa}-\beta\geq 0.
\end{equation}
Let now $u_1=(ax+by)^+$ and $u_2=(ax+by)^-,$ defined for every $a,b\in\mathbb{R},$ such that $a\not =0$ or $b\not=0.$ Then, We invoke Lemma  \ref{lemma-integral-quotient-gradient-x-+}, concluding that
\begin{equation}\label{FFF2}
 \frac{\int_{\partial B_1^{{\mathbb{H}^1}}(0)}\frac{\mid\nabla_{\mathbb{H}^1} u_1(\kappa)\mid^2}{\sqrt{x^2+ y^2}}dP^{B_1^{{\mathbb{H}^1}}(0)}_{\mathbb{H}^1}(\kappa)}{\int_{B_1^{{\mathbb{H}^1}}(0)}\frac{\mid\nabla_{\mathbb{H}^1} u_1(\kappa)\mid^2}{|\kappa |_{\mathbb{H}^1}^{2}}d\kappa}+\frac{\int_{\partial B_1^{{\mathbb{H}^1}}(0)}\frac{\mid\nabla_{\mathbb{H}^1} u_2(\kappa)\mid^2}{\sqrt{x^2+ y^2}}dP^{B_1^{{\mathbb{H}^1}}(0)}_{\mathbb{H}^1}(\kappa)}{\int_{B_1^{{\mathbb{H}^1}}(0)}\frac{\mid\nabla_{\mathbb{H}^1} u_2(\kappa)\mid^2}{|\kappa |_{\mathbb{H}^1}^{2}}d\kappa}=4.
\end{equation}
 Thus, if $\beta>4,$ then $u_1=(ax+by)^+$ and $u_2=(ax+by)^-$ satisfy the hypotheses of  Lemma \ref{lemmaapplicant}, but  $\frac{J_{\beta,\mathbb{H}^1}'(1)}{J_{\beta,\mathbb{H}^1}(1)}< 0$ when $J_{\beta,\mathbb{H}^1}$ is tested on $u_1=(ax+by)^+$ and $u_2=(ax+by)^-.$ Hence, in order to preserve the increasing monotonicity of $J_{\beta,\mathbb{H}^1},$ from (\ref{FFF2}) we are forced to suppose that $\beta\leq 4.$ 
\end{proof}

\end{document}